\documentclass{amsart}
\usepackage{amsmath,amsthm,amssymb}
\usepackage{enumerate}
\usepackage{amsfonts}
\usepackage{calligra}
\usepackage{appendix}

\def\R{\mathbb{R}}
\def\N{\mathbb{N}}

\def\T{\mathbb{T}}
\def\supp{\operatorname{supp}}
\def\dive{\operatorname{div}}

\newcommand{\norm}[1]{\left\lVert#1\right\rVert}
\newcommand{\abs}[1]{\left\lvert#1\right\rvert}

\newcommand{\defeq}{\mathrel{:\mkern-0.25mu=}}
\newcommand{\eqdef}{\mathrel{=\mkern-0.25mu:}}

\newtheorem{thm}{Theorem}[section]
\newtheorem{cor}[thm]{Corollary}
\newtheorem{prop}[thm]{Proposition}
\newtheorem{lem}[thm]{Lemma}

\theoremstyle{definition}
\newtheorem{defin}[thm]{Definition}
\newtheorem{rem}[thm]{Remark}

\numberwithin{equation}{section}

\begin{document}

\title{Magnetic helicity and subsolutions in ideal MHD}


\author{Daniel Faraco}
\address{Departamento de Matem\'{a}ticas \\ Universidad Aut\'{o}noma de Madrid, E-28049 Madrid, Spain; ICMAT CSIC-UAM-UC3M-UCM, E-28049 Madrid, Spain}
\email{daniel.faraco@uam.es}
\thanks{D.F. was partially supported by ICMAT Severo Ochoa projects SEV-2011-0087 and SEV-2015-556, the grants MTM2014-57769-P-1, MTM2014-57769-P-3 and MTM2017-85934-P (Spain) and the ERC grant 307179-GFTIPFD. S.L. was supported by the ERC grant 307179-GFTIPFD}

\author{Sauli Lindberg}
\address{Departamento de Matem\'{a}ticas \\ Universidad Aut\'{o}noma de Madrid, E-28049 Madrid, Spain; ICMAT CSIC-UAM-UC3M-UCM, E-28049 Madrid, Spain}
\email{sauli.lindberg@uam.es}
\thanks{}

\subjclass[2010]{35Q35, 76W05, 76B03}

\keywords{Ideal magnetohydrodynamics, Tartar framework, magnetic helicity, mean-square magnetic potential}
\date{}

\begin{abstract}
We show that ideal 2D MHD does not possess weak solutions (or even subsolutions) with compact support in time and non-trivial magnetic field. We also show that the $\Lambda$-convex hull of ideal MHD has empty interior in both 2D and 3D; this is seen by finding suitable $\Lambda$-convex functions. As a consequence we show that mean-square magnetic potential is conserved in 2D by subsolutions and weak limits of solutions in the physically natural energy space $L^\infty_t L^2_x$, and in 3D we show the conservation of magnetic helicity by $L^3$-integrable subsolutions and weak limits of solutions. However, in 3D the $\Lambda$-convex hull is shown to be large enough that nontrivial smooth, compactly supported strict subsolutions exist.
\end{abstract}
\maketitle

\section{Introduction}
Magnetohydrodynamics (MHD in short) couples Maxwell's equations with hydrodynamics to study the macroscopic behaviour of electrically conducting fluids such as plasmas and liquid metals (see \cite{Davidson} and \cite{ST}). On the $n$-dimensional torus $\T^n = [0,1]^n$ the Cauchy problem for \emph{ideal (nonviscous) MHD} consists of the equations
\begin{align}
& \partial_t u + \dive(u \otimes u - b \otimes b) + \nabla \Pi = 0, \label{MHD} \\
& \partial_t b + \dive(b \otimes u - u \otimes b) = 0, \label{MHD2} \\
& \dive u = \dive b = 0, \label{MHD3} \\
& u(\cdot,0) = u_0, \; b(\cdot,0) = b_0, \label{MHD4} \\
& \int_{\T^n} u(x,t) \, dx = \int_{\T^n} b(x,t) \, dx = 0 \qquad \text{for almost every } t \in [0,T[, \label{MHD5}
  \end{align}
where $T > 0$, $u \in L^2_{loc}(\T^n \times [0,T[; \R^n)$ is the velocity field, $b \in L^2_{loc}(\T^n \times [0,T[; \R^n)$ is the magnetic field, $\Pi \in L^1_{loc}(\T^n \times [0,T[)$ is the total pressure and the initial datas $u_0,b_0 \in L^2(\T^n;\R^n)$ are divergence-free. Equations \eqref{MHD}--\eqref{MHD4} are understood in the sense of distributions, that is,
\[\int_0^T \int_{\T^n} \left[ u \cdot \partial_t \varphi + (u \otimes u - b \otimes b) : \operatorname{D} \varphi + \Pi \dive \varphi \right] + \int_{\T^n} u_0 \cdot \varphi(\cdot,0) = 0,\]
\[\int_0^T \int_{\T^n} [b \cdot \partial_t \varphi + (b \otimes u - u \otimes b) : \operatorname{D} \varphi] + \int_{\T^n} b_0 \cdot \varphi(\cdot,0) = 0,\]
\[\int_0^T \int_{\T^n} (u \cdot \nabla) \varphi = \int_0^T \int_{\T^n} (b \cdot \nabla) \varphi = 0\]
for all $\varphi \in C_c^\infty(\T^n \times [0,T[;\R^n)$. An analogous definition, without condition \eqref{MHD5}, is given in $\R^n$ with $u,b \in L^2_{loc}(\R^n \times [0,T[; \R^n)$ and test functions $\varphi \in C_c^\infty(\R^n \times [0,T[;\R^n)$. The Euler equations are a special case of ideal MHD where $b \equiv 0$.

Starting from the pioneering work of De Lellis and Sz\'{e}kelyhidi on the Euler equations in \cite{DLS09}, the Tartar framework has been used to show the existence of solutions with compact support in time for many equations of hydrodynamics. Such patological weak solutions were already known to exist in the case of Euler equations (see \cite{Scheffer} and \cite{Shnirelman}) but the method of De Lellis and Sz\'{e}kelyhidi is very robust and many equations in hydrodynamics are amenable to it and its ramifications, see \cite{BLFNL}, \cite{BSV}, \cite{BV}, \cite{CCF}, \cite{Chiodaroli}, \cite{CDLK}, \cite{CFK}, \cite{CM}, \cite{CS}, \cite{CFG}, \cite{IV}, \cite{KY}, \cite{LLX}, \cite{Shvydkoy}, \cite{Szekelyhidi} and \cite{TZ}. The regularity of solutions is related to the celebrated Onsager conjecture (see \cite{Buckmaster}, \cite{BDLIS15}, \cite{BDLS13}, \cite{BDLS16}, \cite{BDLSV}, \cite{CET}, \cite{Dan}, \cite{DS}, \cite{DLS09}, \cite{DLS10}, \cite{DLS13}, \cite{DLS14}, \cite{DLS16}, \cite{Eyi}, \cite{Ise13}, \cite{Ise16}, \cite{Ise17}, \cite{IO16}, \cite{Onsager} and \cite{Shvydkoy Onsager}).

In this work we instead use the Tartar framework to show conservation of integral quantities under weak assumptions.  In the Tartar framework the ideal MHD equations \eqref{MHD}--\eqref{MHD3} are decoupled into a set of linear partial differential equations and a pointwise constraint: the equations are
\begin{align}
& \dive u = \dive b = 0, \label{Linearized MHD1} \\
& \partial_t u + \dive S = 0, \label{Linearized MHD2} \\
& \partial_t b + \dive A = 0, \label{Linearized MHD3}
  \end{align}
where $S$ takes values in the set $\mathcal{S}^{n \times n}$ of symmetric matrices and $A$ takes values in the set $\mathcal{A}^{n \times n}$ of antisymmetric matrices and the constraint set is
\begin{align*}
K &\defeq \{(u,b,S,A) \in \R^n \times \R^n \times \mathcal{S}^{n \times n} \times \mathcal{A}^{n \times n} \colon S = u \otimes u - b \otimes b + \Pi I, \; \Pi \in \R, \\
& A = b \otimes u - u \otimes b\}.
\end{align*}
The wave cone $\Lambda$ is, loosely speaking, the set of directions in which 1-dimensional oscillating waves satisfy equations \eqref{MHD5}--\eqref{Linearized MHD3}, and the $\Lambda$-convex hull $K^\Lambda$ consists of those points that cannot be separated from $K$ by functions that are convex in the directions of $\Lambda$ (see \textsection \ref{The wave cone and the lamination convex hull} for the exact definitions). A solution of \eqref{MHD5}--\eqref{Linearized MHD3} that takes values in $K^\Lambda$ at a.e. $(x,t) \in \T^n \times ]0,T[$ is called a \emph{subsolution}.

It is a classical fact that smooth solutions of \eqref{MHD}--\eqref{MHD5} conserve magnetic helicity in 3D and mean-square magnetic potential in 2D. We show that the conservation is in fact very robust and extends even to subsolutions and weak limits of solutions. This phenomenon is seen here as a reflection of the shape of the $\Lambda$-convex hull of the constraint set $K$. As a corollary, in 2D, there exist no weak solutions of the MHD equations \eqref{MHD}--\eqref{MHD5} that have compact support in time and non-trivial magnetic field $b$.

In 2D we consider weak solutions of \eqref{MHD}--\eqref{MHD5} where $u$ and $b$ belong to the natural energy space $L^\infty_t L^2_x(\T^2 \times [0,T[;\R^2)$. By redefining $u$ and $b$ in a set of times of measure zero we may then assume that $u,b \in C_w([0,T[;L^2(\T^2;\R^2))$ (where we denote $v \in C_w([0,T[;L^2(\T^2;\R^2))$ when $v \in L^\infty_t L^2_x(\T^2 \times [0,T[;\R^2)$ and $t_j \to t$ implies $v(\cdot,t_j) \rightharpoonup v(\cdot,t)$ in $L^2(\T^2;\R^2)$ for every $t \in [0,T[$); this can be seen by a modification of \cite[Lemmas 2.2 and 2.4]{Gal}. In our result $\Psi \in C_w([0,T[; W^{1,2}(\T^2))$ is the unique stream function of $b$ that satisfies $\int_{\T^2} \Psi(x,t) \, dx = 0$ for every $t \in [0,T[$ (see Lemma \ref{Stream function lemma}). The result is new also for weak solutions of \eqref{MHD}--\eqref{MHD3}.

\begin{thm} \label{Mean-square magnetic potential preservation theorem}
If $u,b \in C_w([0,T[;L^2(\T^2;\R^2))$, $S \in L^1_{loc}(\T^2 \times [0,T[;\mathcal{S}^{2 \times 2})$ and $A \in L^1_{loc}(\T^2 \times [0,T[;\mathcal{A}^{2 \times 2})$ form a solution of \eqref{MHD4}--\eqref{Linearized MHD3} such that $(u,b,S,A)(x,t) \in K^\Lambda$ a.e. $(x,t) \in \T^2 \times ]0,T[$, then the mean-square magnetic potential $\int_{\T^2} \abs{\Psi(x,t)}^2 dx$ is constant in $t$.
\end{thm}

Theorem \ref{Mean-square magnetic potential preservation theorem} is a reflection of the existence of a suitable $\Lambda$-convex function which shows that if $(u,b,S,A)(x,t) \in K^\Lambda$ a.e., then $A = b \otimes u - u \otimes b$ and so \eqref{MHD2} is satisfied (see \textsection \ref{Emptiness of the interior of the hull in 2D}). Another main ingredient of the proof is the fact that the evolution of $\Psi$ can be described in terms of the Jacobian of a two-dimensional map: $\partial_t \Psi - J_{(\Psi,\Phi)} = 0$, where $\Phi$ is the stream function of $u$ (see Lemma \ref{Time evolution lemma for Psi}). Thus we can use the Hardy space theory of Jacobians: first, if $u$ and $b$ are smooth, we use an integration by parts to compute
\[\partial_t \frac{1}{2} \int_{\T^2} |\Psi(x,t)|^2 dx = \int_{\mathbb{T}^2} \Psi(x,t) J_{(\Psi,\Phi)}(x,t) \, dx = - \int_{\mathbb{T}^2} \Phi(x,t) J_{(\Psi,\Psi)}(x,t) \, dx = 0;\]
in the general case where $u,b \in C_w([0,T[;L^2(\T^2;\R^2))$ we use Sobolev embedding to get $\Psi \in L^\infty_t \operatorname{BMO}_r(\mathbb{T}^2 \times [0,T[)$ and the $\mathcal{H}^1$ regularity theory of Coifman, Lions, Meyer and Semmes from \cite{CLMS} to get $J_{(\Psi,\Phi)} \in L^\infty_t \mathcal{H}^1_z(\T^2 \times [0,T[)$, and Fefferman's classical $\mathcal{H}^1$--$\operatorname{BMO}$ duality result from \cite{FS} is then used to provide the necessary approximation argument. The proof is presented in \textsection \ref{Conservation of the mean-square magnetic potential} where we also show that the mean-square magnetic potential is conserved by weak limits of solutions (see Remark \ref{Remark on weak limits of solutions}). Theorem \ref{Mean-square magnetic potential preservation theorem} implies Corollary \ref{MHD 2D corollary} by using the Poincar\'{e} inequality at every $t \in [0,T[$ to estimate $\int_{\T^2} \abs{b(x,t)}^2 dx = \int_{\T^2} \abs{\nabla \Psi(x,t)}^2 dx \ge C \int_{\T^2} \abs{\Psi(x,t)}^2 dx$.

\begin{cor} \label{MHD 2D corollary}
Suppose that $u,b \in C_w([0,T[;L^2(\T^2;\R^2))$ satisfy \eqref{MHD2}--\eqref{MHD5}. Then either $b \equiv 0$ or there exists $C > 0$ such that $\int_{\T^2} \abs{b(x,t)}^2 dx \ge C$ for every $t \in [0,T[$.
\end{cor}

Corollary \ref{MHD 2D corollary} rules out convex integration solutions that are in the energy space and compactly supported in time (aside from Euler solutions for which $b \equiv 0$). Note also that by Corollary \ref{MHD 2D corollary} it is not possible to construct weak solutions which dissipate magnetic energy to zero -- in fact, one can easily replace $T$ by $\infty$ in Theorem \ref{Mean-square magnetic potential preservation theorem} and Corollary \ref{MHD 2D corollary}.

In 3D, in contrast to Corollary \ref{MHD 2D corollary}, Bronzi, Lopes Filho and Nussenzveig Lopes used in \cite{BLFNL} convex integration in the Tartar framework to show that there exist infinitely many bounded weak solutions of \eqref{MHD}--\eqref{MHD3} that are of the symmetry reduced form
\begin{equation} \label{Symmetry reduced form}
u(x_1,x_2,x_3,t) = (u_1(x_1,x_2,t),u_2(x_1,x_2,t),0), \; b(x_1,x_2,x_3,t) = (0,0,b_3(x_1,x_2,t)),
\end{equation}
compactly supported in time and with $u,b \not\equiv 0$. The solutions were obtained cleverly via two-dimensional Euler equations with a passive tracer (see \textsection \ref{Discussion of the solutions of Bronzi & al.}). Note, however, that the solutions are independent of the $x_3$ variable. In particular, the construction of \cite{BLFNL} does not give solutions of MHD that are compactly supported in space.

We put the results of \cite{BLFNL} in context. In 3D we may in \eqref{Linearized MHD3} identify the antisymmetric matrix $A = [a_{ij}]_{i,j=1}^3$ with $a \defeq (a_{23},a_{31},a_{12}) \in \R^3$ and write \eqref{Linearized MHD3} as
\begin{equation} \label{Linearized MHD 4}
\partial_t b + \nabla \times a  = 0
\end{equation}
(see \textsection \ref{Emptiness of the interior of the hull in 3D} for details on this formalism). We show in Theorem \ref{Empty interior theorem on hull} that $Q(u,b,S,a) \defeq a \cdot b$ is $\Lambda$-affine and vanishes for every $(u,b,S,a) \in K^\Lambda$. This implies that $K^\Lambda$ has empty interior, which would normally make the use of convex integration very difficult (for a situation where a non-linear pointwise constraint is succesfully understood see \cite{MS99}). Under the restrictions \eqref{Symmetry reduced form} the constraint $a \cdot b = 0$ does not cause trouble precisely because $a$ and $b$ take values in orthogonal subspaces of $\R^3$ (see \textsection \ref{Discussion of the solutions of Bronzi & al.}).

We use the Tartar framework to show the following 3D analogue of Theorem \ref{Magnetic helicity preservation theorem}, where $\Psi$ is a vector potential of $b$, i.e. $\nabla \times \Psi = b$ (see Lemma \ref{Helmholtz-Hodge lemma}).

\begin{thm} \label{Magnetic helicity preservation theorem}
Suppose that $u,b \in L^3(\T^3 \times ]0,T[;\R^3)$, $S \in L^1_{loc}(\T^3 \times ]0,T[;\mathcal{S}^{3 \times 3})$ and $a \in L^{3/2}(\T^3 \times ]0,T[;\R^3)$ form a solution of \eqref{MHD5}--\eqref{Linearized MHD2},\eqref{Linearized MHD 4} that takes values in $K^\Lambda$. Then the magnetic helicity $\int_{\T^3} \Psi(x,t) \cdot b(x,t) dx$ is constant a.e. in $t$.
\end{thm}

For solutions of the linearized MHD equations \eqref{MHD5}--\eqref{Linearized MHD2},\eqref{Linearized MHD 4} on $\T^3$, the time evolution of magnetic helicity is given by
\begin{equation} \label{Time evolution of magnetic helicity}
\partial_t \int_{\T^3} \Psi(x,t) \cdot b(x,t) \, dx = - 2 \int_{\T^3} a(x,t) \cdot b(x,t) \, dx,
\end{equation}
and Theorem \ref{Magnetic helicity preservation theorem} follows from the fact that $a \cdot b$ vanishes in $K^\Lambda$. The details of the proof are presented in \textsection \ref{Conservation of magnetic helicity by subsolutions}.

As a quadratic $\Lambda$-affine quantity $a \cdot b$ is weakly continuous (see Lemma \ref{Weak compactness lemma}), and so, using \eqref{Time evolution of magnetic helicity}, we also show in Theorem \ref{Conservation of magnetic helicity by weak limits of solutions} that magnetic helicity is conserved by weak limits of $L^3$ solutions of 3D MHD. The same principle is behind the conservation of mean-square magnetic potential by weak limits of solutions in 2D; there the quadratic $\Lambda$-affine quantity is $b \times u = J_{(\Psi,\Phi)}$.

Despite the remarkable robustness of magnetic helicity conservation, compactly supported convex integration solutions cannot, at this point, be ruled out in 3D MHD. In fact, the $\Lambda$-convex hull $K^\Lambda$ turns out to have non-empty relative interior (relative to the constraint $a \cdot b = 0$). This result, recorded in the following theorem, is the technically most difficult part of the paper and requires careful analysis of the interplay between $K$ and $\Lambda$.

\begin{thm} \label{Non-empty relative interior theorem}
In 3D MHD, $\operatorname{int}(K^\Lambda) = \emptyset$. However, the point $(0,0,0,0)$ belongs to the relative interior of $K^\Lambda$ in the set $\{(u,b,S,a) \colon b \cdot a = 0\} \subset \R^3 \times \R^3 \times \mathcal{S}^{3 \times 3} \times \R^3$.
\end{thm}

Theorem \ref{Non-empty relative interior theorem} suggests defining strict subsolutions in analogy to incompressible Euler equations and many other equations of fluid dynamics (see \textsection \ref{Definition of subsolutions} for the precise definition). The following 3D result is in stark contrast to Theorem \ref{Mean-square magnetic potential preservation theorem}.

\begin{thm} \label{Subsolution theorem}
In $\R^3$ the ideal MHD equations have strict subsolutions $(u,b,S,a) \in C_c^\infty(\R^3 \times \R; \R^3 \times \R^3 \times \mathcal{S}^{3 \times 3} \times \R^3)$ with $u,b \not\equiv 0$.
\end{thm}

Three interesting problems arise from our work. The first one is whether the emptiness of the interior of the 3D $\Lambda$-convex hull can be overcome and some variant of the convex integration approach works in this setting. The second one is an explicit computation of the $\Lambda$-convex hulls and the question whether they coincide with the relaxation of MHD. The third one is whether, in 2D, every Cauchy data determines a unique $b$ in the energy space. The last question is
a special case of the general problem whether there exists a dichotomy
between the existence of convex integration solutions and the
uniqueness of the Cauchy problem.

\section{Background}

In this section we give results on stream functions and vector and scalar potentials, and we present proofs in cases where they are difficult to find in the literature. We also discuss previously known results on conserved integral quantities in MHD.

\subsection{Stream functions in 2D} \label{Stream functions and vector potentials}
In the following standard lemma we find stream functions for solutions of the ideal MHD equations on the torus $\T^2$. The lemma concerns time-dependent mappings in Bochner spaces, and for more information on Bochner spaces we refer to \cite{HVNVW}. 

\begin{lem} \label{Stream function lemma}
If $v \in C_w([0,T[;L^2(\T^2;\R^2))$ satisfies $\dive v = 0$ and $\int_{\T^2} v(x,t) \, dx = 0$ for every $t \in [0,T[$, then there exists a unique function $\Theta \in C_w([0,T[;W^{1,2}(\T^2)) \cap C([0,T[;L^2(\T^2))$ with $-\nabla^\perp \Theta \defeq (\partial_2 \Theta,-\partial_1 \Theta) = v$ and $\int_{\T^2} \Theta(x,t) \, dx = 0$ for every $t \in [0,T[$.
\end{lem}

\begin{proof}[Sketch of proof]
At every $t \in [0,T[$ we get $\dive v(\cdot,t) = 0$ (by integrating $v$ against suitable test functions of the form $\varphi_1(x) \varphi_2(t)$). The existence and uniqueness of $\Theta(\cdot,t)$ is proven by standard Fourier analysis (similar to the proof of Lemma \ref{Helmholtz-Hodge lemma}), and Poincar\'{e} inequality gives $\norm{\Theta(\cdot,t)}_{W^{1,2}(\T^2)} \lesssim \norm{v(\cdot,t)}_{L^2}$. The mapping $t \mapsto \Theta(\cdot,t) \colon [0,T[ \to W^{1,2}(\T^2)$ is strongly measurable by the linearity and boundedness of the operator that maps $v(\cdot,t)$ to $\Theta(\cdot,t)$ and the strong measurability of $t \mapsto v(\cdot,t)$. Now a standard application of the Rellich-Kondrachov Theorem gives $\Theta \in C([0,T[;L^2(\T^2))$.
\end{proof}

One of the main ideas behind Theorem \ref{Mean-square magnetic potential preservation theorem} is that the time evolution of the stream function of $b$ is governed by a Jacobian determinant which is an $\mathcal{H}^1$ integrable quantity.

\begin{lem} \label{Time evolution lemma for Psi}
Suppose that $u,b \in C_w([0,T[;L^2(\T^2;\R^2))$ satisfy \eqref{MHD2}--\eqref{MHD5} and that $\Phi, \Psi \in C_w([0,T[;W^{1,2}(\T^2))$ are the stream functions of $u$ and $b$ given by Lemma \ref{Stream function lemma}. Then
\begin{equation} \label{Time evolution of psi in 2D}
\partial_t \Psi - J(\Psi,\Phi) = 0
\end{equation}
in $\T^2 \times ]0,T[$.
\end{lem}

Note that Lemma \ref{Time evolution lemma for Psi} does not require \eqref{MHD} as an assumption. Before presenting the proof of Lemma \ref{Time evolution lemma for Psi} we fix, for the rest of this article, a mollifier $\chi \in C_c^\infty(\T^2 \times \R)$ of the tensor product form $\chi(x,t) = \chi_x(x) \chi_t(t)$, where $\int_{\T^2} \chi_x(x) \, dx = \int_{-\infty}^\infty \chi_t(t) \, dt = 1$. We assume that $\chi$ is even and $\text{supp}(\chi) \subset \T^2 \times ]-1,1[$. When $\delta > 0$, we define $\chi^\delta(x,t) \defeq \delta^{-3} \chi(x/\delta,t/\delta)$. We also denote, e.g., $\Psi_\delta \defeq \Psi * \chi^\delta$, where $\Psi$ is the stream function of $b$. Note that for every $\delta > 0$ and every $t \in [\delta,T-\delta]$,
\begin{equation} \label{Vanishing integral mean 2}
\int_{\T^2} \Psi_\delta(x,t) \, dx
= \int_{t-\delta}^{t+\delta} \chi^\delta_t(t-s) \int_{\T^2} \Psi(y,s) \int_{\T^2} \chi_x^\delta(x-y) \, dx \, dy ds = 0.
\end{equation}

\begin{proof}
The second MHD equation \eqref{MHD2} can be written as
\[-\nabla^\perp \partial_t \Psi + \dive( \nabla^\perp \Psi \otimes \nabla^\perp \Phi - \nabla^\perp \Phi \otimes \nabla^\perp \Psi) = 0.\]
Note that $-\nabla^\perp J(\Psi,\Phi) = \dive( \nabla^\perp \Psi \otimes \nabla^\perp \Phi - \nabla^\perp \Phi \otimes \nabla^\perp \Psi)$, which implies that
\[\partial_t \Psi - J(\Psi,\Phi) \eqdef g \in \mathcal{D}'(\T^2 \times ]0,T[)\]
satisfies $\nabla g = 0$. By \eqref{Vanishing integral mean 2} and a similar formula for $J(\Psi,\Phi)$ we therefore get $g_\delta = 0$ in $\T^2 \times ]\delta,T-\delta[$ for every $\delta \in ]0,T/2[$, and so $g = 0$.
\end{proof}

\begin{rem}
As pointed out to the authors by L\'{a}szl\'{o} Sz\'{e}kelyhidi Jr., in the 2D Euler equations the vorticity satisfies an identity similar to \eqref{Time evolution of psi in 2D}: when $\phi$ is the stream function and $\omega \defeq \nabla^\perp \cdot u$ is the vorticity of the velocity $u$, we have $\partial_t \omega - J(\omega,\phi) = 0$ (see \cite[p. 442]{VN}).
\end{rem}

\subsection{Vector and scalar potentials in 3D} \label{Vector and scalar potentials in 3D}
We will present 3D analogues of Lemmas \ref{Stream function lemma} and \ref{Time evolution lemma for Psi}, and we first recall the Helmholtz-Hodge decomposition in $L^p$ spaces on $\T^3$ (see \cite[Theorem 2.28]{RRS}).

\begin{lem} \label{Helmholtz-Hodge at one time point}
Suppose $1 < p < \infty$. Then every $v \in L^p(\T^3;\R^3)$ with $\int_{\T^3} v(x) \, dx = 0$ can be written uniquely as
\[v = u + \nabla g,\]
where $u \in L^p(\T^3,\R^3)$ satisfies $\dive u = 0$, $\int_{\T^3} u(x) \, dx = 0$ and $\int_{\T^3} \abs{u(x)}^p dx \lesssim_p \int_{\T^3} \abs{v(x)}^p dx$ whereas $g \in W^{1,p}(\T^3)$ satisfies $\int_{\T^3} g(x) \, dx = 0$ and $\int_{\T^3} \abs{\nabla g(x)}^p dx \lesssim_p \int_{\T^3} \abs{v(x)}^p dx$.
\end{lem}

For time-dependent mappings in $L^p(\T^3 \times ]0,T[;\R^3)$, a Helmholtz-Weyl decomposition in suitable Bochner spaces can be proved by using Lemma \ref{Helmholtz-Hodge at one time point} at a.e. time $t \in ]0,T[$; strong measurability follows from the fact that the operators $v \mapsto u$ and $v \mapsto \nabla g$ of Lemma \ref{Helmholtz-Hodge at one time point} are linear. However, the particular form of the Helmholtz-Weyl decomposition needed here appears difficult to find in the literature, and we therefore sketch a proof. Lemma \ref{Helmholtz-Hodge lemma} also provides information on the divergence-free component that is crucial in \textsection \ref{Conservation of magnetic helicity by subsolutions}.

\begin{lem} \label{Helmholtz-Hodge lemma}
Let $1 < p < \infty$. If $v \in L^p(\T^3 \times ]0,T[;\R^3)$ satisfies $\int_{\T^3} v(x,t) \, dx = 0$ a.e. $t \in ]0,T[$, then $v$ can be uniquely written as
\[v = \nabla \times \Theta + \nabla g,\]
where $\Theta \in L^p_t W^{1,p}_x(\T^3 \times ]0,T[;\R^3)$ satisfies $\int_{\T^3} \Theta(x,t) \, dx = 0$ a.e. $t \in ]0,T[$ and $\dive \Theta = 0$ whereas $g \in L^p_t W^{1,p}_x(\T^3 \times ]0,T[)$ satisfies $\int_{\T^3} g(x,t) \, dx = 0$ a.e. $t \in ]0,T[$. Furthermore,
\[\norm{\Theta}_{L^p_t W^{1,p}_x} + \norm{g}_{L^p_t W^{1,p}_x} \lesssim_p \norm{v}_{L^p}.\]
If $\dive v = 0$, then $g = 0$, and if $\nabla \times v = 0$, then $\Theta = 0$.
\end{lem}

As in 2D, we fix for the rest of the article an even mollifier $\chi \in C_c^\infty(\T^3 \times \R)$ of the tensor product form $\chi(x,t) = \chi^x(x) \chi^t(t)$, where $\int_{\T^3} \chi^x(x) \, dx = \int_{-\infty}^\infty \chi^t(t) \, dt = 1$ and $\text{supp}(\chi) \subset \T^3 \times ]-1,1[$. When $\delta > 0$, we denote $\chi^\delta(x,t) \defeq \delta^{-4} \chi(x/\delta,t/\delta)$.

\begin{proof}
We get $\Theta$ and $g$ as limits of smooth mappings. Let $0 < \delta < \epsilon < T-\epsilon$. We denote $v_\delta \defeq v * \chi^\delta = \sum_{k \in \mathbb{Z}^3 \setminus \{0\}} c^k_\delta(t) e^{2 \pi i k \cdot x} \in C^\infty(\T^3 \times ]\epsilon,T-\epsilon[; \R^3)$ and fix $t \in \R$. We use standard Fourier analysis and the uniqueness statement of Lemma \ref{Helmholtz-Hodge at one time point} to write
\begin{align*}
   v_\delta(x,t)
&= \sum_{k \in \mathbb{Z}^3 \setminus \{0\}} \left( c_k^\delta(t) - \frac{c^k_\delta(t) \cdot k}{\abs{k}^2} k \right) e^{2 \pi i k \cdot x}
+ \sum_{k \in \mathbb{Z}^3 \setminus \{0\}} \frac{c^k_\delta(t) \cdot k}{\abs{k}^2} k \,  e^{2 \pi i k \cdot x} \\
&\eqdef \sum_{k \in \mathbb{Z}^3 \setminus \{0\}} a^k_\delta(t) e^{2 \pi i k \cdot x}
+ \sum_{k \in \mathbb{Z}^3 \setminus \{0\}} b^k_\delta(t)  e^{2 \pi i k \cdot x} = u^\delta(x,t) + \nabla g^\delta(x,t)
\end{align*}
for every $x \in \T^3$.

Now $T^\delta(x,t) \defeq \sum_{k \in \mathbb{Z}^3 \setminus \{0\}} \abs{2 \pi k}^{-2} a^k_\delta(t) e^{2 \pi i k \cdot x}$ is a solution of the Poisson equation $-\Delta T^\delta(\cdot,t) = u^\delta(\cdot,t)$ and furthermore $\dive T^\delta(\cdot,t) = 0$. Thus
\[\nabla \times (\nabla \times T^\delta)(\cdot,t) = \nabla \dive T^\delta(\cdot,t) - \Delta T^\delta(\cdot,t) = u^\delta(\cdot,t).\]
We set $\Theta^\delta(\cdot,t) \defeq \nabla \times T^\delta(\cdot,t)$ so that $u^\delta(\cdot,t) = \nabla \times \Theta^\delta(\cdot,t)$. The natural norm bound $\norm{\partial_i \partial_j T^\delta(\cdot,t)}_{L^p(\T^3)} \lesssim_p \norm{\Delta T^\delta}_{L^p(\T^3)}$ for all $i,j \in \{1,2,3\}$ (see \cite[Theorem B.7]{RRS}), combined with Lemma \ref{Helmholtz-Hodge at one time point} and the Poincar\'{e} inequality, yields
\[\norm{\Theta^\delta(\cdot,t)}_{W^{1,p}(\T^3)} + \|g^\delta(\cdot,t)\|_{W^{1,p}(\T^3)} \lesssim_p \norm{v_\delta(\cdot,t)}_{L^p(\T^3)}.\]
Now a decomposition $v = \nabla \times \Theta + \nabla g$ in $\T^3 \times ]\epsilon,T-\epsilon[$ is found via standard limiting arguments.

Uniqueness of $g$ in $\T^3 \times ]\epsilon,T-\epsilon[$ follows from Lemma \ref{Helmholtz-Hodge at one time point}. For the uniqueness of $\Theta$ suppose $\tilde{\Theta}$ is another vector potential that satisfies the conditions of Lemma \ref{Helmholtz-Hodge lemma} in $\T^3 \times ]\epsilon,T-\epsilon[$. Then $\Delta (\Theta - \tilde{\Theta}) = - \nabla \times (\nabla \times (\Theta - \tilde{\Theta})) = 0$ and $\int_{\T^3} (\Theta(x,t) - \tilde{\Theta}(x,t)) \, dx = 0$ a.e. $t \in ]\epsilon,T-\epsilon[$, which leads to $\Theta = \tilde{\Theta}$ since the periodic extension of $(\Theta-\tilde{\Theta})_\delta$ is bounded and harmonic. This gives the unique Hodge decomposition in $\T^3 \times ]0,T[$ with the desired norm bounds.

In order to finish the proof of the lemma suppose that $\dive v = 0$. Then $-\Delta g = \dive (v - \nabla \times \Theta) = 0$ and $\int_{\T^3} g(x,t) \, dx = 0$ for a.e. $t \in ]0,T[$, which implies that $g = 0$. Similarly, $\nabla \times v = 0$ leads to $-\Delta \Theta = \nabla \times (\nabla \times \Theta) = \nabla \times (v - \nabla g) = 0$, yielding $\Theta = 0$.
\end{proof}

Lemma \ref{Helmholtz-Hodge lemma} implies the following Poincar\'{e}-type lemma with norm bounds for solutions of linearized 3D MHD. 

\begin{lem} \label{Time evolution of vector potential}
Suppose $b \in L^3(\T^3 \times ]0,T[;\R^3)$ and $a \in L^{3/2}(\T^3 \times ]0,T[;\R^3)$ satisfy
\begin{align*}
\dive b &= 0, \\
\partial_t b + \nabla \times a &= 0.
\end{align*}
Then there exist unique $\Psi \in L^3_t W^{1,3}_x(\T^3 \times ]0,T[;\R^3)$ and $g \in L^{3/2}_t W^{1,3/2}_x(\T^3 \times ]0,T[)$ such that
\[b = \nabla \times \Psi \quad \text{and} \quad \partial_t \Psi + a - \int_{\T^3} a(y,\cdot) dy = \nabla g\]
with $\int_{\T^3} \Psi(x,t) \, dx = 0$ and $\int_{\T^3} g(x,t) \, dx = 0$ for a.e. $t \in ]0,T[$ and $\dive \Psi = 0$. Furthermore,
\[\norm{\Psi}_{L^3_t W^{1,3}_x} \lesssim \norm{b}_{L^3} \quad \text{and} \quad \norm{\partial_t \Psi}_{L^{3/2}} + \norm{g}_{L^{3/2}_t W^{1,3/2}_x} \lesssim \norm{a}_{L^{3/2}}.\]
\end{lem}

\begin{proof}
The only part of Lemma \ref{Time evolution of vector potential} that does not follow immediately from Lemma \ref{Helmholtz-Hodge lemma} is the claim 
that in the decomposition $a - \int_{\T^3} a(y,\cdot) \, dy = \nabla \times \Theta + \nabla g$ we have $\nabla \times \Theta = - \partial_t \Psi$. In order to show this let $0 < \epsilon < T/2$. Whenever $0 < \delta < \epsilon$, we write $a_\delta - (\int_{\T^3} a(y,\cdot) \, dy)_\delta = \nabla \times \Theta_\delta + \nabla g_\delta$. On the other hand, from the equation $\nabla \times (\partial_t \Psi + a - \int_{\T^3} a(y,\cdot) \, dy) = 0$ and Lemma \ref{Helmholtz-Hodge lemma} we get $\partial_t \Psi_\delta + a_\delta - (\int_{\T^3} a(y,\cdot) \, dy)_\delta = \nabla \tilde{g}$. By Lemma \ref{Helmholtz-Hodge lemma}, $\dive \Psi = 0$, and so $\dive \partial_t \Psi_\delta = 0$. Since $\int_{\T^3} \partial_t \Psi_\delta(x,t) \, dx = 0$ for every $t \in ]\epsilon,T-\epsilon[$, the uniqueness of the Helmholtz-Weyl decomposition in Lemma \ref{Helmholtz-Hodge at one time point} implies that  $\nabla \times \Theta_\delta = - \partial_t \Psi_\delta$ and $g_\delta = \tilde{g}$. Thus $\nabla \times \Theta = \lim_{\delta \searrow 0} \nabla \times \Theta_\delta = \lim_{\delta \searrow 0} -\partial_t \Psi_\delta = - \partial_t \Psi$ in $\mathcal{D}'(\T^3 \times ]\epsilon,T-\epsilon[;\R^3)$, which proves the claim.
\end{proof}

\subsection{Classically conserved quantities of ideal MHD} \label{Classically conserved quantities of ideal MHD}
We define three classically conserved quantities of ideal 3D MHD on the torus $\T^3$; analogous definitions, under suitable assumptions, are available in $\R^3$.

\begin{defin}
Suppose that $u,b \in C^\infty(\T^3 \times ]0,T[; \R^3)$ and $\Pi \in C^\infty(\T^3 \times ]0,T[)$ satisfy the ideal MHD equations and that $\Psi \in C^\infty(\T^3 \times ]0,T[;\R^3)$ satisfies $\nabla \times \Psi = b$ and $\int_{\T^3} \Psi(x,t) \, dx = 0$ for every $t \in ]0,T[$. The \emph{total energy}, \emph{magnetic helicity} and \emph{cross helicity} of $(u,b,\Pi)$ are defined as
\begin{align*}
&\frac{1}{2} \int_{\T^3} (\abs{u(x,t)}^2+\abs{b(x,t)}^2) \, dx, \\
&\int_{\T^3} \Psi(x,t) \cdot b(x,t) \, dx, \\
&\int_{\T^3} u(x,t) \cdot b(x,t) .\, dx.
\end{align*}
\end{defin}

All three quantities defined above are conserved in time by smooth solutions. For results on total energy and cross helicity conservation for weak solutions we refer to \cite{CKS}, \cite{KL} and \cite{Yu}. Conservation of the magnetic helicity was shown in \cite{CKS} for $u \in C([0,T]; B_{3,\infty}^{\alpha_1}(\T^3;\R^3))$ and $b \in C([0,T]; B_{3,\infty}^{\alpha_2}(\T^3;\R^3))$ with $\alpha_1 + 2 \alpha_2 > 0$. In \cite{KL}, Kang and Lee showed magnetic helicity conservation under the assumption that $u,b \in C_w([0,T];L^2(\R^3;\R^3)) \cap L^3_t L^3_x(\R^3 \times ]0,T[; \R^3)$, although the assumptions on $\Psi$ are not made completely explicit. Theorems \ref{Magnetic helicity preservation theorem} and \ref{Conservation of magnetic helicity by weak limits of solutions} generalise magnetic helicity conservation to subsolutions and weak limits of $L^3$ solutions. In 2D, magnetic helicity has the following natural counterpart.

\begin{defin}
Suppose $u,b \in C^\infty(\T^2 \times ]0,T[; \R^2)$ and $\Pi \in C^\infty(\T^2 \times ]0,T[)$ satisfy the ideal MHD equations and $\Psi \in C^\infty(\T^2 \times ]0,T[;\R^2)$ is the stream function $\Psi$ of $b$ with $\int_{\T^2} \Psi(x,t) \, dx = 0$ for every $t \in ]0,T[$. The \emph{mean-square magnetic potential} of $(u,b,\Pi)$ is defined as $\int_{\T^2}  \abs{\Psi(x,t)}^2 dx$.
\end{defin}

In \cite{CKS}, the conservation of the mean-square magnetic potential is shown for $u$ and $\Psi$ in the Besov spaces $C([0,T]; B^{\alpha_1}_{3,\infty}(\T^2,\R^2))$ and $C([0,T]; B^{\alpha_2}_{3,\infty}(\T^2,\R^2))$ respectively, where $\alpha_1 + 2 \alpha_2 > 1$. In Theorem \ref{Mean-square magnetic potential preservation theorem} we prove conservation under the assumption that $u,b \in C_w([0,T[;L^2(\T^2 ;\R^2))$.

\section{Ideal MHD in the Tartar framework} \label{Ideal MHD in the Tartar framework}

This article is devoted to studying the MHD equations in the Tartar framework, and in this section we recall many of the relevant definitions. We also compute the wave cone in both 2D and 3D.
The computation of the wave cone and Theorem \ref{Theorem on nonempty relative interior} are also partial results towards the existence of compactly supported convex integration solutions of 3D MHD. For more information on the Tartar framework see \cite{Tartar79}, \cite{Tartar83} and in the context of fluid dynamics see \cite{CFG}, \cite{DLS12}, \cite{Szekelyhidi}.

\subsection{Linearization of MHD in Els\"asser variables}
In order to facilitate the ensuing computations we use the Els\"asser (characteristic) variables $z^{\pm} \defeq u \pm b$ to rewrite the MHD equations in terms of $(z^+,z^-,\Pi)$ in the symmetric form
\begin{equation} \label{MHD in Elsasser variables}
\begin{cases}
\partial_t z^+ + \dive(z^+ \otimes z^- + \Pi I) &= 0, \\
\partial_t z^- + \dive(z^- \otimes z^+ + \Pi I) &= 0, \\
\dive z^{\pm} &= 0
  \end{cases}
\end{equation}
(for the main motivation, however, see Remark \ref{Remark on Elsasser variables}). In order to  write \eqref{Linearized MHD1}--\eqref{Linearized MHD3} in this formalism we denote $M \defeq S + A \colon \R^n \times \R \to \R^{n \times n}$ so that $S = (M+M^T)/2$ and $A = (M-M^T)/2$.

\begin{defin}
The linear partial differential operator $\mathcal{L} \colon \mathcal{D}'(\R^n \times \R; \R^n \times \R^n \times \R^{n \times n}) \to \mathcal{D}'(\R^n \times \R; \R^n \times \R^n \times \R \times \R)$ is defined by
\begin{equation} \label{Definition of L}
\mathcal{L}(z^+,z^-,M) \defeq \left[ \begin{array}{c}
\partial_t z^+ + \dive M \\
\partial_t z^- + \dive M^T \\
\dive z^+ \\
\dive z^-
  \end{array} \right].
\end{equation}
On the torus $\T^n$ we add to $\mathcal{L}$ the two components $t \mapsto \int_{\T^n} z^\pm(x,t) \, dx$.
\end{defin}
We decouple \eqref{MHD in Elsasser variables} into the linear equation $\mathcal{L}(z^+,z^-,M) = 0$ and the pointwise constraint that $(z^+,z^-,M)$ takes values in the set
\begin{equation} \label{Definition of K}
K \defeq \{(z^+,z^-,M) \colon M = z^+ \otimes z^- + \Pi I, \; \Pi \in \R\}.
\end{equation}
We also define normalized versions of $K$ by setting
\[K_{r,s} \defeq \{(z^+,z^-,M) \colon \abs{z^+} = r, \abs{z^-} = s, \; M = z^+ \otimes z^-+ \Pi I, \; \abs{\Pi} \le rs\}\]
for all $r,s > 0$.

\begin{rem} \label{Remark on Elsasser variables}
The use of Els\"{a}sser variables is natural in the context of convex integration for MHD. Indeed, it is still an open problem whether in $\R^3$ there exist compactly supported weak solutions of MHD which do not conserve cross helicity. This suggests, in analogy to the work done on many other equations of fluid dynamics, an attempt to prescribe the total energy and the cross helicity (densities) at every time $t$. One would achieve this if one could prescribe $\abs{z^+}$ and $\abs{z^-}$, as $(\abs{u}^2 + \abs{b}^2)/2 = (\abs{z^+}^2 + \abs{z^-}^2)/4$ and $u \cdot b = (\abs{z^+}^2-\abs{z^-}^2)/4$. In contrast, $u \cdot b$ obviously cannot be written in terms of $\abs{u}$ and $\abs{b}$.
\end{rem}

\subsection{The wave cone and the lamination convex hull} \label{The wave cone and the lamination convex hull}
\emph{Plane waves} are one-dimensional oscillations $(x,t) \mapsto h((x,t) \cdot (\xi,c)) (\alpha,\beta,M)$ with $(\alpha,\beta,M) \in \R^n \times \R^n \times \R^{n \times n}$, $(\xi,c) \in (\R^n \times \R) \setminus \{0\}$ and $h \colon \R \to \R$, and the wave cone, defined below, gives the set of directions of plane waves satisfying the linearized MHD equation $\mathcal{L}(z^+,z^-,M) = 0$. Here and in the sequel we will often use the isomorphism
\[(\alpha,\beta,M) \mapsto \left[ \begin{array}{cc}
M & \alpha \\
\beta^T & 0
  \end{array} \right] \colon \R^n \times \R^n \times \R^{n \times n} \to \R^{(n+1) \times (n+1)}_0,\]
where $\R^{(n+1) \times (n+1)}_0 \defeq \{A \in \R^{(n+1) \times (n+1)} \colon a_{n+1,n+1} = 0\}$. The linearized MHD equation $\mathcal{L}(z^+,z^-,M) = 0$ can be written in divergence form as
\[\dive_{x,t} \left[ \begin{array}{cc}
M & z^+ \\
(z^-)^T & 0
  \end{array} \right]
= \dive_{x,t} \left[ \begin{array}{cc}
M^T & z^- \\
(z^+)^T & 0
  \end{array} \right] = 0.\]

\begin{defin}
The \emph{wave cone} of ideal MHD is
\[\Lambda_0 = \left\{ V \in \R^{4 \times 4}_0 \colon \exists \left[ \begin{array}{c}
                       \xi \\
                       c
                     \end{array} \right] \in \R^{n+1} \setminus \{0\} \text{ such that } V \left[ \begin{array}{c}
                       \xi \\
                       c
                     \end{array} \right] = V^T \left[ \begin{array}{c}
                       \xi \\
                       c
                     \end{array} \right] = 0 \right\}.\]
We also denote
\[\Lambda = \left\{ V \in \R^{4 \times 4}_0 \colon \exists \left[ \begin{array}{c}
                       \xi \\
                       c
                     \end{array} \right] \in (\R^n \setminus \{0\}) \times \R \text{ such that } V \left[ \begin{array}{c}
                       \xi \\
                       c
                     \end{array} \right] = V^T \left[ \begin{array}{c}
                       \xi \\
                       c
                     \end{array} \right] = 0 \right\}.\]
\end{defin}

We will use the subset $\Lambda$ of the wave cone $\Lambda_0$ as it is much easier to use in convex integration. For convex integration it is important to gain enough information about the lamination convex hull of $K$, which we next define. Given a set $Y \subset \R^{(n+1) \times (n+1)}_0$ we denote $Y^{0,\Lambda} \defeq Y$ and define inductively
\[Y^{N+1,\Lambda} \defeq Y^{N,\Lambda} \cup \{\lambda V + (1-\lambda) W \colon \lambda \in [0,1], \; V, W \in Y^{N,\Lambda}, \; V-W \in \Lambda\}\]
for all $N \in \N_0$.

\begin{defin}
When $Y \subset \R^{(n+1) \times (n+1)}_0$, the \emph{lamination convex hull} of $Y$ (with respect to $\Lambda$) is
\[Y^{lc,\Lambda} \defeq \bigcup_{N \ge 0} Y^{N,\Lambda}.\]
\end{defin}

In addition to the lamination convex hull, another, potentially larger, hull is used in convex integration theory. In order to define it recall the following

\begin{defin}
A function $f \colon \R^{(n+1) \times (n+1)}_0 \to \R$ is said to be \emph{$\Lambda$-convex} if the function $t \mapsto f(V + tW) \colon \R \to \R$ is convex for every $V \in \R^{(n+1) \times (n+1)}_0$ and every $W \in \Lambda$.
\end{defin}

While the lamination convex hull is defined by taking convex combinations, the $\Lambda$-convex hull of $Y \subset \R^{(n+1) \times (n+1)}_0$ is defined as the set of points that cannot be separated from $Y$ by $\Lambda$-convex functions.

\begin{defin}
When $Y \subset \R^{(n+1) \times (n+1)}_0$, the $\Lambda$-\emph{convex hull} $Y^\Lambda$ consists of points $W \in \R^{(n+1) \times (n+1)}_0$ with the following property: if $f \colon \R^{(n+1) \times (n+1)}_0 \to \R$ is $\Lambda$-convex and $f|_Y \le 0$, then $f(W) \le 0$.
\end{defin}

\subsection{Computation of the wave cone in 3D MHD} \label{Computation of the wave cone}
Recall that $(\alpha,\beta,M) \in \Lambda$ if and only if there exists $(\xi,c) \in (\R^3 \setminus \{0\}) \times \R$ such that
\begin{equation} \label{Wave cone conditions}
\begin{cases}
M \xi + c \alpha & = 0, \\
M^T \xi + c \beta & = 0, \\
\alpha \cdot \xi & = 0, \\
\beta \cdot \xi & = 0.
  \end{cases}
\end{equation}
We split the computation of $\Lambda$ into a few special cases. When $\alpha, \beta \in \R^3$ with $\alpha \times \beta \neq 0$,  the nine tensor products $\alpha \otimes \alpha$, $\alpha \otimes \beta$, $\alpha \otimes (\alpha \times \beta)$, $\beta \otimes \alpha$, $\beta \otimes \beta$, $\beta \otimes (\alpha \times \beta)$, $(\alpha \times \beta) \otimes \alpha$, $(\alpha \times \beta) \otimes \beta$, $(\alpha \times \beta) \otimes (\alpha \times \beta)$ form a basis of $\R^{3 \times 3}$, and we write elements of $\R^{3 \times 3}$ in the form $M = c_{11} \alpha \otimes \alpha + \cdots + c_{33} (\alpha \times \beta) \otimes (\alpha \times \beta)$.

\vspace{0.4cm}
\begin{lem} \label{Wave cone proposition 1}
When $\alpha \times \beta \neq 0$, the triple $(\alpha,\beta,M) \in \Lambda$ if and only if
\[
M = c_{11} \alpha \otimes \alpha + c_{12} \alpha \otimes \beta + c_{13} [\alpha \otimes (\alpha \times \beta) + (\alpha \times \beta) \otimes \beta]
+ c_{21} \beta \otimes \alpha + c_{22} \beta \otimes \beta\]
for some $c_{11},\dots,c_{22} \in \R$.
\end{lem}

\begin{proof}
Suppose \eqref{Wave cone conditions} are satisfied. Then $\xi$ is necessarily a nonzero multiple of $\alpha \times \beta$, and we may assume that $\xi = \alpha \times \beta$. By using the fact that $\alpha \cdot \alpha \times \beta = \beta \cdot \alpha \times \beta = 0$ we get
\[\begin{array}{lcl}
& & M (\alpha \times \beta) \\
&=& c_{13} [\alpha \otimes (\alpha \times \beta)] (\alpha \times \beta) + c_{23} [\beta \otimes (\alpha \times \beta)] (\alpha \times \beta) \\
&+& c_{33} [(\alpha \times \beta) \otimes (\alpha \times \beta)] (\alpha \times \beta) \\
&=& c_{13} \abs{\alpha \times \beta}^2 \alpha + c_{23} \abs{\alpha \times \beta}^2 \beta + c_{33} \abs{\alpha \times \beta}^2 \alpha \times \beta
  \end{array}\]
and
\[\begin{array}{lcl}
& & M^T (\alpha \times \beta) \\
&=& c_{31} [(\alpha \times \beta) \otimes \alpha]^T (\alpha \times \beta) + c_{32} [(\alpha \times \beta) \otimes \beta]^T (\alpha \times \beta) \\
&+& c_{33} [(\alpha \times \beta) \otimes (\alpha \times \beta]^T (\alpha \times \beta) \\
&=& c_{31} \abs{\alpha \times \beta}^2 \alpha + c_{32} \abs{\alpha \times \beta}^2 \beta + c_{33} \abs{\alpha \times \beta}^2 \alpha \times \beta.
  \end{array}\]
Now \eqref{Wave cone conditions} implies that $c_{13} = c_{32}$ and $c_{23} = c_{31} = c_{33} = 0$ but the coefficients $c_{11}$, $c_{12}$, $c_{21}$ and $c_{22}$ are free. For the converse choose $(\xi,c) = (\alpha \times \beta, - c_{13} \abs{\alpha \times \beta}^2)$ in \eqref{Wave cone conditions}.
\end{proof}

\begin{lem} \label{Wave cone proposition 2}
Suppose $\alpha \neq 0$ and $k \in \R$. We have $(\alpha,k\alpha,M) \in \Lambda$ if and only if $M$ is of the form
\begin{equation} \label{Form of M}
\begin{array}{lcl}
M &=& c_{11} \alpha \otimes \alpha + c_{12} [\alpha \otimes \gamma + k \gamma \otimes \alpha] \\
&+& c_{13} \alpha \otimes (\alpha \times \gamma) + c_{31} (\alpha \times \gamma) \otimes \alpha + c_{33} (\alpha \times \gamma) \otimes (\alpha \times \gamma)
\end{array}
\end{equation}
for some $\gamma \neq 0$ with $\gamma \perp \alpha$ and $c_{11},\dots,c_{33} \in \R$. In particular, $(\alpha,k\alpha,f \otimes \alpha + \alpha \otimes g) \in \Lambda$ for every $f,g \in \R^3$.
\end{lem}

\begin{proof}[Proof]
Let $0 \neq \xi \perp \alpha$. Denote $\gamma = \xi$ and write
\[M = c_{11} \alpha \otimes \alpha + c_{12} \alpha \otimes \gamma + \cdots + c_{33} (\alpha \times \gamma) \otimes (\alpha \times \gamma).\]
Now \eqref{Wave cone conditions} is equivalent to
\begin{align*}
& M \gamma + c \alpha
= c_{12} \abs{\gamma}^2 \alpha + c_{22} \abs{\gamma}^2  \gamma + c_{32} \abs{\gamma}^2 (\alpha \times \gamma) + c \alpha = 0, \\
& M^T \gamma + k c \alpha
= c_{21} \abs{\gamma}^2 \alpha + c_{22} \abs{\gamma}^2 \gamma + c_{23} \abs{\gamma}^2 (\alpha \times \gamma) + k c \alpha = 0,
\end{align*}
which proves the first claim.

In order to demonstrate the second claim we fix $f,g \in \R^3$. If $\alpha \times (f - kg) \neq 0$, we set $\gamma \defeq \alpha \times (f - kg)/\abs{\alpha \times (f - kg)}$. Writing $f = \langle f, \alpha/\abs{\alpha} \rangle \alpha/\abs{\alpha} + \langle f, \gamma \rangle \gamma + \langle f, \alpha/\abs{\alpha} \times \gamma \rangle \alpha/\abs{\alpha} \times \gamma$ and similarly for $g$ we note that $\langle f, \gamma \rangle = k \langle \gamma, g \rangle$ and so $f \otimes \alpha + \alpha \otimes g$ is of the form \eqref{Form of M}, where $c_{12} = \langle \gamma, g \rangle$. If, on the other hand, $\alpha \times f = k \alpha \times g \neq 0$, we set $\gamma = \alpha \times f/\abs{\alpha \times f}$. As above, $f \otimes \alpha + \alpha \otimes g$ is of the form \eqref{Form of M} with $c_{12} = 0$. Next, if $\alpha \times f = k \alpha \times g = 0$ and $k \neq 0$, then $f \otimes \alpha + \alpha \otimes g$ is simply a multiple of $\alpha \otimes \alpha$. Finally, if $\alpha \times f = 0$ and $k = 0$, we choose any unit vector $\gamma \perp \alpha$; \eqref{Form of M} is satisfied with $c_{12} = \langle \gamma, g \rangle$.
\end{proof}

All that is left is the characterization of the case $\alpha = 0$. When $\alpha = 0$ and $\beta \neq 0$, we have a similar situation as in the preceding lemma:

\begin{lem} \label{Wave cone proposition 3}
When $\beta \neq 0$, we have $(0,\beta,M) \in \Lambda$ if and only if
\[
M = c_{12} \gamma \otimes \beta + c_{22} \beta \otimes \beta + c_{23} \beta \otimes (\gamma \times \beta) + c_{32} (\gamma \times \beta) \otimes \beta + c_{33} (\gamma \times \beta) \otimes (\gamma \times \beta)\]
for some $\gamma \neq 0$ with $\gamma \perp \beta$ and $c_{12},\dots,c_{32} \in \R$. In particular, $(0,\beta, f \otimes \beta + \beta \otimes g) \in \Lambda$ for all $f,g \in \R^3$.
\end{lem}

\begin{proof}[Proof]
For the first claim let us denote $0 \neq \xi = \gamma \perp \beta$. We write $M = c_{11} \gamma \otimes \gamma + c_{12} \gamma \otimes \beta + \cdots + c_{32} (\gamma \times \beta) \otimes \beta + c_{33} (\gamma \times \beta) \otimes (\gamma \times \beta)$. Now \eqref{Wave cone conditions} is equivalent to
\begin{align*}
    M \gamma
&= c_{11} \abs{\gamma}^2 \gamma + c_{21} \abs{\gamma}^2 \beta + c_{31} \abs{\gamma}^2 \gamma \times \beta = 0, \\
    M^T \gamma + c \beta
&= c_{11} \abs{\gamma}^2 \gamma + c_{12} \abs{\gamma}^2 \beta + c_{13} \abs{\gamma}^2 \gamma \times \beta + c \beta = 0,
\end{align*}
proving the first claim of the lemma. The second claim is proved as in the case $k = 0$ of Lemma \ref{Wave cone proposition 2}.
\end{proof}

Whenever $\alpha \neq 0$ or $\beta \neq 0$, the conditions $(\alpha,\beta,M) \in \Lambda$ and $(\alpha,\beta,M) \in \Lambda_0$ are in fact equivalent. However, when $\alpha = \beta = 0$, belonging to $\Lambda$ is a much more restrictive condition.

\begin{lem} \label{Wave cone proposition 4}
If $M \in \R^{3 \times 3}$, we have $(0,0,M) \in \Lambda$ if and only if there exist orthonormal vectors $f_1,f_2 \in \R^3$ and coefficients $c_{11}, c_{12}, c_{21}, c_{22} \in \R$ such that $M = \sum_{i,j=1}^2 c_{ij} f_i \otimes f_j$.
\end{lem}

\begin{proof}[Proof]
If $(0,0,M) \in \Lambda$, choose $(\xi,c) \in (\R^3 \setminus \{0\}) \times \R$ such that $M \xi = M^T \xi = 0$. Form an orthonormal basis $\{f_1,f_2,f_3\}$ of $\R^3$, where $f_3 = \xi/\abs{\xi}$, and write $M = \sum_{i,j=1}^3 c_{ij} f_i \otimes f_j$. Then $M \xi = \sum_{i=1}^3 c_{i,3} f_i = 0$ and $M^T \xi = \sum_{j=1}^3 c_{3,j} f_j = 0$ imply that $M = \sum_{i,j=1}^2 c_{ij} f_i \otimes f_j$. The converse is proved by setting $\xi = f_1 \times f_2$ and choosing any $c \in \R$.
\end{proof}

\subsection{The wave cone in 2D MHD}
In 2D MHD the $\Lambda$-convex hull $K^{\Lambda}$ has empty interior. However, $K^{\Lambda}$ is strictly larger than $K$ itself; these results are shown in Proposition \ref{Empty interior proposition}. We first compute the wave cone in this section.

Recall that the subset of the wave cone that we use is
\[\Lambda = \left\{ V \in \R^{3 \times 3}_0 \colon \exists \left[ \begin{array}{c}
                       \xi \\
                       c
                     \end{array} \right] \in (\R^2 \setminus \{0\}) \times \R \text{ such that } V \left[ \begin{array}{c}
                       \xi \\
                       c
                     \end{array} \right] = V^T \left[ \begin{array}{c}
                       \xi \\
                       c
                     \end{array} \right] = 0 \right\}.\]
If $\alpha \in \R^2 \setminus \{0\}$, we may write any $M \in \R^{2 \times 2}$ as a linear combination of four basis elements of $\R^{2 \times 2}$: $M = c_{11} \alpha \otimes \alpha + c_{12} \alpha \otimes \alpha^\perp + c_{21} \alpha^\perp \otimes \alpha + c_{22} \alpha^\perp \otimes \alpha^\perp$, where $\alpha^\perp = (-\alpha_2,\alpha_1)$. Now a necessary condition for $(\alpha,\beta,M) \in \Lambda$ is that $0 \neq \xi \perp \{\alpha,\beta\}$ so that $\beta = k \alpha$ with $k \in \R$. Lemma \ref{2D Wave cone proposition 1} is proved by an easy adaptation of the proofs of Lemmas \ref{Wave cone proposition 2}--\ref{Wave cone proposition 4}.

\begin{lem} \label{2D Wave cone proposition 1}
When $\alpha \neq 0$ and $k \in \R$, the triple $(\alpha,k \alpha,M) \in \Lambda$ if and only if
\[M = c_{11} \alpha \otimes \alpha + c_{12} [\alpha \otimes \alpha^\perp + k \alpha^\perp \otimes \alpha],\]
where $c_{11},c_{12} \in \R$. When $\beta \neq 0$, we have $(0,\beta,M) \in \Lambda$ if and only if
\[M = c_{11} \beta \otimes \beta + c_{21} \beta^\perp \otimes \beta,\]
where $c_{11},c_{12} \in \R$. Furthermore, $(0,0,M) \in \Lambda$ if and only if $M = \gamma \otimes \gamma$ for some $\gamma \in \R^2$.
\end{lem}

\section{Proof of Theorem \ref{Mean-square magnetic potential preservation theorem}}
Theorem \ref{Mean-square magnetic potential preservation theorem} is proved in this section. In Proposition \ref{Empty interior proposition} we find a suitable $\Lambda$-convex function which gives us crucial information about the shape of the $\Lambda$-convex hull $K^\Lambda$, and in \textsection \ref{Conservation of the mean-square magnetic potential} we finish the proof of Theorem \ref{Mean-square magnetic potential preservation theorem}.

\subsection{Emptiness of the interior of the hull in 2D} \label{Emptiness of the interior of the hull in 2D}
The constraint set $K$ for 2D ideal MHD was defined in \eqref{Definition of K}. The following result on the $\Lambda$-convex hull of $K$ was the initial motivation behind Theorem \ref{Mean-square magnetic potential preservation theorem} and Corollary \ref{MHD 2D corollary}.

\begin{prop} \label{Empty interior proposition}
Whenever $(\alpha,\beta,\alpha \otimes \beta + N) \in K^\Lambda$, the matrix $N$ is symmetric. In particular, $K^\Lambda$ has empty interior. However, $K^\Lambda \setminus K \neq \emptyset$.
\end{prop}

\begin{proof}
The idea of the proof of the emptiness statement is to construct a $\Lambda$-convex function $f$ such that $f$ vanishes on $K$ but $f(\alpha,\beta,M) > 0$ whenever $M - \alpha \otimes \beta$ is non-symmetric.

We define
\begin{equation} \label{Lambda convex function in 2D}
f(\alpha,\beta,M) \defeq |(m_{12} - \alpha_1 \beta_2) - (m_{21} - \alpha_2 \beta_1)|^2
\end{equation}
intending to show that when $V \defeq (\alpha,\beta,M) \in \R^2 \times \R^2 \times \R^{2 \times 2}$ and $W \defeq (\gamma, k \gamma, N) \in \Lambda$, the function
\[t \mapsto g(t) \defeq f(V+tW) \colon \R \to \R\]
is convex. (A similar proof works for triples of the form $(0,\gamma,N) \in \Lambda$.)

With $V$ and $W$ fixed and $t \in \R$, we get
\[\begin{array}{lcl}
& & g(t) - g(0) \\
&=& |(m_{12} + t n_{12} - (\alpha_1 + t \gamma_1)(\beta_2+tk \gamma_2)) \\
&-& (m_{21} + t n_{21} - (\alpha_2 + t \gamma_2)(\beta_1+tk \gamma_1))|^2 \\
&-& |(m_{12} - \alpha_1 \beta_2) - (m_{21} - \alpha_2 \beta_1)|^2 \\
&=& |(m_{12} - \alpha_1 \beta_2) - (m_{21} - \alpha_2 \beta_1) + dt|^2 - |(m_{12} - \alpha_1 \beta_2) - (m_{21} - \alpha_2 \beta_1)|^2 \\
&=& ct + |d t|^2,
  \end{array}\]
since the terms $\pm t^2 k \gamma_1 \gamma_2$ cancel out. Thus $g''(t) = 2 d^2 \ge 0$ for all $t$, so that $g$ is convex.

We then show that $K^\Lambda \setminus K \neq \emptyset$. When $\alpha, \beta, \gamma, \delta \in \R^2$, we have $( \alpha, \beta, \alpha \otimes \beta) - ( \gamma, \delta, \gamma \otimes \delta) \in \Lambda$ precisely when $\alpha, \beta, \gamma,$ and $\delta$ lie on the same line (see Lemma \ref{Restriction on next action} and its proof for an analogous statement in 3D). Now $\alpha^\pm = (\pm 3,1)$ and $\beta^\pm = (\pm 2,1)$ belong to the horizontal line through $(0,1)$ and $(\alpha^\pm,\beta^\pm,\alpha^\pm \otimes \beta^\pm) \in K$, yet
\[\frac{(\alpha^+,\beta^+,\alpha^+ \otimes \beta^+)}{2} + \frac{(\alpha^-,\beta^-,\alpha^- \otimes \beta^-)}{2} = \left( (0,1), (0,1), \left[ \begin{array}{cc}
                      6 & 0 \\
                      0 & 1
                    \end{array} \right]
 \right) \notin K.\]
The claim $K^\Lambda \setminus K \neq \emptyset$ also follows from the fact that the $\Lambda$-convex hull for Euler equations is non-trivial. Indeed, when $z^+ = z^-$ and $M$ is symmetric, the linearized MHD equations $\mathcal{L}(z^+,z^+,M) = 0$ (see \eqref{Definition of L}) reduce to the linearized Euler equations.
\end{proof}

When written in terms of $(u,b,S,A)$, the $\Lambda$-convex function in \eqref{Lambda convex function in 2D} is of the form $\tilde{f}(u,b,S,a) \defeq 4 \abs{a_{12} - b \times u}^2$ and Proposition \ref{Empty interior proposition} obtains the following form.

\begin{cor} \label{Subsolutions are solutions for b in 2D}
If $(u,b,S,A) \in K^\Lambda$, then $A = b \otimes u - u \otimes b$.
\end{cor}

By Corollary \ref{Subsolutions are solutions for b in 2D}, the assumptions of Theorem \ref{Mean-square magnetic potential preservation theorem} imply that \eqref{MHD2} holds. Thus Theorem \ref{Mean-square magnetic potential preservation theorem} will be proved once we show that the mean-square magnetic potential is conserved by solutions of \eqref{MHD2}--\eqref{MHD5}.

\subsection{Conservation of the mean-square magnetic potential} \label{Conservation of the mean-square magnetic potential}
The aim of this section is to finish the proof of Theorem \ref{Mean-square magnetic potential preservation theorem}; our proof is reminiscent of that of \cite[Theorem 1.4]{IV}. We use the $\mathcal{H}^1$ regularity theory of Coifman, Lions, Meyer and Semmes from \cite{CLMS}, more precisely the following adaptation of the classical Wente inequality to the torus $\T^2$ (see \cite[Theorem A.1]{FMS}).
\begin{lem} \label{Jacobian estimate}
When $(f_1,f_2,f_3) \in W^{1,2}(\T^2,\R^3)$, we have
\begin{align} \label{Jacobian inequality}
\int_{\T^2} f_1(x) J_{(f_2,f_3)}(x) \, dx
&\lesssim \norm{f_1}_{\operatorname{BMO}(\T^2)} \norm{J_{(f_2,f_3)}}_{\mathcal{H}^1(\T^2)} \notag \\
&\lesssim \norm{\nabla f_1}_{L^2(\T^2)} \norm{\nabla f_2}_{L^2(\T^2)} \norm{\nabla f_3}_{L^2(\T^2)}.
\end{align}
\end{lem}

The left-hand side of \eqref{Jacobian inequality} can be understood in terms of $\mathcal{H}^1$--$\text{BMO}$ duality, but we will in fact use estimate \eqref{Jacobian inequality} only in cases where the left-hand side is Lebesgue integrable. For the proof of Theorem \ref{Mean-square magnetic potential preservation theorem} we fix a mollifier $\chi \in C_c^\infty(\T^2 \times \R)$ as in \textsection \ref{Stream functions and vector potentials}. Since $\chi$ is even, we have $\int_\epsilon^{T-\epsilon} \int_{\T^2} f(x,t) g_\delta(x,t) \, dx \, dt = \int_\epsilon^{T-\epsilon} \int_{\T^2} f_\delta(x,t) g(x,t) \, dx \, dt$ for all $f \in L^1(\T^3 \times ]0,T[; \R^3)$ and $g \in L^\infty(\T^3 \times ]0,T[; \R^3)$ whenever $0 < \delta < \epsilon < T-\epsilon$.

\begin{proof}[Proof of Theorem \ref{Mean-square magnetic potential preservation theorem}]
Since $\Psi \in C([0,T[;L^2(\T^2))$, it suffices to show that
\begin{equation} \label{Conservation of norm}
\int_0^T \partial_t \eta(t) \int_{\T^2} \abs{\Psi(x,t)}^2 dx \, dt = 0
\end{equation}
for every $\eta \in C_c^\infty(]0,T[)$. We fix $\eta$ and choose $\epsilon > 0$ such that $\supp(\eta) \subset [2\epsilon,T-2\epsilon]$. Now $\norm{\Psi - \Psi_\delta}_{L^2(\T^2 \times (\epsilon,T-\epsilon))} \to 0$ and $\norm{\nabla \Psi - \nabla \Psi_\delta}_{L^2(\T^2 \times (\epsilon,T-\epsilon))} \to 0$ as $\delta \searrow 0$, and so \eqref{Time evolution of psi in 2D} yields
\begin{align*}
    \int_\epsilon^{T-\epsilon} \partial_t \eta(t) \int_{\T^2} \abs{\Psi(x,t)}^2 dx \, dt
&= \lim_{\delta \searrow 0} \int_\epsilon^{T-\epsilon} \partial_t \eta(t) \int_{\T^2} \abs{\Psi_\delta(x,t)}^2 dx \, dt \\
&= 2 \lim_{\delta \searrow 0} \int_\epsilon^{T-\epsilon} \eta(t) \int_{\T^2} \Psi_\delta(x,t) [J_{(\Psi,\Phi)}]_\delta(x,t) \, dx \, dt.
\end{align*}

When $\delta > 0$ is small, we write
\[[J_{(\Psi,\Phi)}]_\delta = [J_{(\Psi - \Psi_\delta,\Phi)}]_\delta
+ ([J_{(\Psi_\delta,\Phi)}]_\delta - J_{(\Psi_\delta,\Phi)}) + J_{(\Psi_\delta,\Phi)}\]
and estimate the resulting integrals separately. First, we use Lemma \ref{Jacobian estimate}, H\"older's inequality and Young's integral inequality to estimate
\begin{align*}
&\hspace{0.478cm} \abs{ \int_\epsilon^{T-\epsilon} \eta(t) \int_{\T^2} \Psi_\delta(x,t) [J_{(\Psi-\Psi_\delta,\Phi)}]_\delta(x,t) \, dx \, dt } \\
&= \abs{ \int_\epsilon^{T-\epsilon} \int_{\T^2} [\eta \Psi_\delta]_\delta(x,t) J_{(\Psi-\Psi_\delta,\Phi)}(x,t) \, dx \, dt } \\
&\lesssim \int_\epsilon^{T-\epsilon} \norm{[\eta \nabla \Psi_\delta]_\delta(\cdot,t)}_{L^2(\T^2)} \norm{\nabla (\Psi-\Psi_\delta)(\cdot,t)}_{L^2(\T^2)} \norm{\nabla \Phi(\cdot,t)}_{L^2(\T^2)} \, dt \\
&\le \norm{[\eta \nabla \Psi_\delta]_\delta}_{L^2(\T^2 \times ]\epsilon,T-\epsilon[)} \norm{\nabla (\Psi-\Psi_\delta)}_{L^2(\T^2 \times ]\epsilon,T-\epsilon[)} \norm{\nabla \Phi}_{L^\infty_t L^2_x(\T^2 \times ]\epsilon,
T-\epsilon[)} \\
&\lesssim \norm{\eta}_{L^\infty} \norm{\nabla \Psi}_{L^2(\T^2 \times ]\epsilon,T-\epsilon[)} \norm{\nabla (\Psi-\Psi_\delta)}_{L^2(\T^2 \times ]\epsilon,T-\epsilon[)} \norm{\nabla \Phi}_{L^\infty_t L^2_x(\T^2 \times ]\epsilon,T-\epsilon[)} \\
&\to 0
\end{align*}
as $\delta \searrow 0$. Similarly,
\begin{align*}
&\hspace{0.478cm} \abs{ \int_\epsilon^{T-\epsilon} \eta(t) \int_{\T^2} \Psi_\delta(x,t) ([J_{(\Psi_\delta,\Phi)}]_\delta - J_{(\Psi_\delta,\Phi)})(x,t) \, dx \, dt } \\
&= \abs{ \int_\epsilon^{T-\epsilon} \int_{\T^2} ([\eta \Psi_\delta]_\delta(x,t) - [\eta \Psi_\delta](x,t)) J_{(\Psi_\delta,\Phi)}(x,t) \, dx \, dt } \\
&\lesssim \int_\epsilon^{T-\epsilon} \norm{[\eta \nabla \Psi_\delta]_\delta(\cdot,t) - \eta(t) \nabla \Psi_\delta(\cdot,t)}_{L^2(\T^2)} \norm{\nabla \Psi_\delta(\cdot,t)}_{L^2(\T^2)} \norm{\nabla \Phi(\cdot,t)}_{L^2(\T^2)} \, dt \\
&\lesssim \norm{[\eta \nabla \Psi_\delta]_\delta - \eta \nabla \Psi_\delta}_{L^2(\T^2 \times ]\epsilon,T-\epsilon[)} \norm{\nabla \Psi_\delta}_{L^2(\T^2 \times ]\epsilon,T-\epsilon[)} \norm{\nabla \Phi}_{L^\infty_t L^2_x(\T^2 \times ]\epsilon,T-\epsilon[)}
\end{align*}
and we get
\begin{align*}
&   \norm{[\eta \nabla \Psi_\delta]_\delta - \eta \nabla \Psi_\delta}_{L^2(\T^2 \times ]\epsilon,T-\epsilon[)} \\
\le & \norm{[\eta \nabla \Psi_\delta - \eta \nabla \Psi]_\delta}_{L^2(\T^2 \times ]\epsilon,T-\epsilon[)}
+ \norm{[\eta \nabla \Psi]_\delta - \eta \nabla \Psi}_{L^2(\T^2 \times ]\epsilon,T-\epsilon[)} \\
+ & \norm{\eta \nabla (\Psi - \Psi_\delta)}_{L^2(\T^2 \times ]\epsilon,T-\epsilon[)} \\
\to & \; 0
  \end{align*}
as $\delta \searrow 0$. Finally,
\[\begin{array}{lcl}
& & \displaystyle \int_\epsilon^{T-\epsilon} \eta(t) \int_{\T^2} \Psi_\delta(x,t) J_{(\Psi_\delta,\Phi)}(x,t) \, dx \, dt \\
&=& \displaystyle - \int_\epsilon^{T-\epsilon} \eta(t) \int_{\T^2} \Phi(x,t) J_{(\Psi_\delta,\Psi_\delta)}(x,t) \, dx \, dt
= 0
\end{array}\]
for every $\delta > 0$, finishing the proof of \eqref{Conservation of norm}.
\end{proof}

It is natural to ask whether an analogue of Theorem \ref{Mean-square magnetic potential preservation theorem} holds in the whole space $\R^2$. However, square integrable divergence-free vector fields do not in general have a square integrable stream function in $\R^2$, a fact that has sometimes been overlooked in the literature. The following simple proposition quantifies this phenomenon in terms of Baire category and shows that the natural analogue of Lemma \ref{Stream function lemma} in $\R^2$ is false.

\begin{prop}
The set $\{\nabla^\perp \Psi \colon \Psi \in W^{1,2}(\R^2)\}$ is of the first Baire category in $\{v \in L^2(\R^2; \R^2) \colon \dive v = 0\}$.
\end{prop}

\begin{proof}
We write $\{\nabla^\perp \Psi \colon \Psi \in W^{1,2}(\R^2)\} = \cup_{k=1}^\infty \{\nabla^\perp \Psi \colon \norm{\Psi}_{W^{1,2}} \le k\}$ intending to prove that each of the closed sets $\{\nabla^\perp \Psi \colon \norm{\Psi}_{W^{1,2}} \le k\}$ has empty interior. Fix $k \in \N$. By the linearity of $\nabla^\perp$, it suffices to show that $\{\nabla^\perp \Psi \colon \norm{\Psi}_{W^{1,2}} \le k\}$ does not contain a ball centered at the origin.

Choose $\Theta \in \dot{W}^{1,2}(\R^2)$ such that $\Theta + C \notin W^{1,2}(\R^2)$ for every $C \in \R$, and denote $v \defeq \nabla^\perp \Theta \in L^2(\R^2;\R^2)$. Seeking contradiction, suppose there exist $\Psi \in W^{1,2}(\R^2)$ and $c \neq 0$ such that $\norm{\Psi}_{W^{1,2}} \le k$ and $\nabla^\perp \Psi = cv$. Then $\nabla^\perp (c\Theta-\Psi) = 0$ and so for some $C \in \R$ we have $\Theta + C = \Psi/c \in W^{1,2}(\R^2)$, which yields the desired contradiction.
\end{proof}

\begin{rem} \label{Remark on weak limits of solutions}
Mean-square magnetic potential is also conserved by weak $L^2$ limits of sequences of solutions that are bounded in $L^\infty_t L^2_x(\T^2 \times [0,T[;\R^2)$; this boils down to the fact that the induction equation $\partial_t b + \dive (b \otimes u - u \otimes b) = 0$ is weakly compact. Indeed, suppose $b^j,u^j \in C_w([0,T[;L^2(\T^2;\R^2))$ solve \eqref{MHD2}--\eqref{MHD5} for every $j\in \mathbb{N}$ with $\sup_{j \in \N} (\|u^j\|_{L^\infty_t L^2_x} + \|b^j\|_{L^\infty_t L^2_x}) < \infty$ and assume $b^j \rightharpoonup b$, $u^j \rightharpoonup u$ in $L^2_t L^2_x(\T^2 \times [0,T[;\R^2)$. Up to passing to a subsequence, the initial datas $b^j_0$ have a weak limit $b_0$ in $L^2(\T^2;\R^2)$. Furthermore, by Lemma \ref{Time evolution lemma for Psi}, $\sup_{j \in \N} \|\partial_t \Psi^j\|_{L^2_t L^1_x} < \infty$, giving $\Psi^j \to \Psi$ in $L^2_t L^2_x(\T^2 \times [0,T[)$ by the Aubin-Lions-Simon Lemma (for a version that suffices for us see e.g. \cite[Lemma 7.7]{Rou}). Now $b^j \times u^j = J(\Psi^j,\Phi^j) \to J(\Psi,\Phi) = b \times u$ in $\mathcal{D}'(\T^2 \times ]0,T[)$, yielding $\partial_t b + \dive (b \otimes u - u \otimes b) = 0$ with $b(\cdot,0) = b_0$. Furthermore, $b,u \in L^\infty_t L^2_x(\T^2 \times [0,T[;\R^2)$, and so again, by a modification of \cite[Lemmas 2.2 and 2.4]{Gal}, $b \in C_w([0,T[;L^2(\T^2;\R^2))$. Now $b$ and $u$ essentially satisfy the assumptions of Theorem \ref{Mean-square magnetic potential preservation theorem} (the weak continuity of $u$ in time is not needed for the conclusion to hold) and so they conserve mean-square magnetic potential in time.
\end{rem}

\section{Proof of Theorem \ref{Magnetic helicity preservation theorem}}

Theorem \ref{Magnetic helicity preservation theorem} gives an analogue of Theorem \ref{Mean-square magnetic potential preservation theorem} in 3D, and the proof is presented in this section. In \textsection \ref{Emptiness of the interior of the hull in 3D} we find a $\Lambda$-affine function which shows that the $\Lambda$-convex hull $K^\Lambda$ has empty interior, and the proof of Theorem \ref{Magnetic helicity preservation theorem} is finished in \textsection \ref{Conservation of magnetic helicity by subsolutions}. Theorem \ref{Magnetic helicity preservation theorem} does not, however, immediately rule out compactly supported convex integration solutions of 3D MHD with $b \not\equiv 0$.

\subsection{Emptiness of the interior of the hull in 3D} \label{Emptiness of the interior of the hull in 3D}
The main aim of this subsection is the construction of a suitable $\Lambda$-affine function. The function gets a much more intuitive form in the formalism with velocity field $u$ and magnetic field $b$ rather than the Els\"{a}sser variables $z^\pm$. Recall from the Introduction that the linearized MHD equations can be written in terms of $(u,b,S,A)$ as
\[\dive_{x,t} \left[ \begin{array}{cc}
S & u \\
u^T & 0
  \end{array} \right] = 0, \qquad
\dive_{x,t} \left[ \begin{array}{cc}
                     A & b \\
                     -b^T & 0
                   \end{array} \right] = 0,\]
so that linearized MHD is divided into the 'symmetric part' satisfied by $u$ and the 'antisymmetric part' satisfied by $b$.

Furthermore, we may identify
\[A = \left[ \begin{array}{ccc}
               0 & a_{12} & -a_{31} \\
               -a_{12} & 0 & a_{23} \\
               a_{31} & -a_{23} & 0
             \end{array} \right]
\cong (a_{23}, a_{31}, a_{12}) \eqdef a \in \R^3.\]
In this identification we have $A \xi = \xi \times a$ for all $\xi \in \R^3$ and $b \otimes u - u \otimes b \cong b \times u$ for all $b,u \in \R^3$. The evolution equation of $b$ thus obtains the more intuitive form
\begin{equation} \label{Maxwell-Faraday}
\partial_t b + \nabla \times a = 0.
\end{equation}
We write $K$ in terms of $u$, $b$, $S$ and $a$ as
\begin{align*}
   K
&= \{(z^+,z^-,M) \colon M = z^+ \otimes z^- + \Pi I, \; \Pi \in \R\} \\
&\cong \{(u,b,S,a) \colon S = u \otimes u - b \otimes b + \Pi I, \; a = b \times u, \; \Pi \in \R\}.
\end{align*}
The wave cone conditions $M \xi + c z^+ = 0$, $M^T \xi + c z^- = 0$, $z^\pm \cdot \xi = 0$ (see \eqref{Wave cone conditions}) translate into
\begin{equation} \label{Wave cone conditions in u and b}
S \xi + c u = 0, \quad \xi \times a + c b = 0, \quad u \cdot \xi = b \cdot \xi = 0.
\end{equation}
Now the sought $\Lambda$-affine function finds a particularly simple form.

\begin{thm} \label{Empty interior theorem on hull}
If $(u,b,S,a) \in \Lambda$ or $(u,b,S,a) \in K^\Lambda$, then $a \cdot b = 0$.
\end{thm}

\begin{proof}
We define
\[Q(u,b,S,a) \defeq a \cdot b,\]
note that $Q|_K = 0$ and set out to show that $Q$ is $\Lambda$-affine. Since $Q$ is quadratic, it suffices to show that $a \cdot b = 0$ for all $(u,b,S,a) \in \Lambda$.

Let $(u,b,S,a) \in \Lambda$ and choose $(\xi,c) \in (\R^3 \setminus \{0\}) \times \R$ such that \eqref{Wave cone conditions in u and b} holds. Thus, writing $b = (b_1,b_2,b_3)$, we have
\begin{equation} \label{Three equations}
\xi_2 a_3 - \xi_3 a_2 + c b_1 = 0, \quad \xi_3 a_1 - \xi_1 a_3 + c b_2  = 0, \quad \xi_1 a_2 - \xi_2 a_1 + c b_3 = 0.
\end{equation}
By multiplying the three equations of \eqref{Three equations} by $a_1, a_2$ and $a_3$ respectively and taking the sum of the left-hand sides we get $c(a_1 b_1 + a_2 b_2 + a_3 b_3) = 0$, that is, $c b \cdot a = 0$, which finishes the proof when $c \neq 0$.

Suppose then $c = 0$. Let $\xi_1 \neq 0$, the cases $\xi_2 \neq 0$ and $\xi_3 \neq 0$ being similar because of symmetry. By assumption, $b \cdot \xi = 0$, and so, by \eqref{Three equations},
\[0 = \frac{a_1}{\xi_1} (\xi_1 b_1 + \xi_2 b_2 + \xi_3 b_3)
= a_1 b_1 + a_2 b_2 + a_3 b_3 = b \cdot a.\]
\end{proof}

\begin{rem}
Note that the variable $a$ is the electric field, \eqref{Maxwell-Faraday} is the Maxwell-Faraday equation and the pointwise constraint $a = b \times u$ is ideal Ohm's law.
\end{rem}

\subsection{Conservation of magnetic helicity by subsolutions and weak limits of solutions} \label{Conservation of magnetic helicity by subsolutions}
The conservation of magnetic helicity by weak solutions $u,b \in L^3_t L^3_x(\T^3 \times ]0,T[;\R^3)$ of \eqref{MHD}--\eqref{MHD3} was proved in \cite{KL}, and we present below a more elementary proof of the result along the lines of \cite[Theorem 1.4]{IV}, while generalizing the result to subsolutions. Recall the even mollifier $\chi \in C_c^\infty(\T^3 \times \R)$ defined in \textsection \ref{Vector and scalar potentials in 3D}.

\begin{proof}[Proof of Theorem \ref{Magnetic helicity preservation theorem}]
Suppose $u,b \in L^3(\T^3 \times ]0,T[;\R^3)$, $S \in L^1_{loc}(\T^3 \times ]0,T[;\mathcal{S}^{3 \times 3})$ and $a \in L^{3/2}(\T^3 \times ]0,T[;\R^3)$ form a solution of \eqref{MHD5}--\eqref{Linearized MHD2} and \eqref{Linearized MHD 4} that takes values in $K^\Lambda$ a.e. Let $\eta \in C_c^\infty(]0,T[)$, so that for $\epsilon > 0$ small enough, $\supp(\eta) \subset ]\epsilon,T-\epsilon[$. By using Lemma \ref{Time evolution of vector potential} and integrating by parts a few times we get
\[\begin{array}{lcl}
& & \displaystyle \int_\epsilon^{T-\epsilon} \partial_t \eta(t) \int_{\T^3} \Psi(x,t) \cdot b(x,t) \, dx \, dt \\
&=& \displaystyle \lim_{\delta \searrow 0} \int_\epsilon^{T-\epsilon} \partial_t \eta(t) \int_{\T^3} \Psi_\delta(x,t) \cdot b_\delta(x,t) \, dx \, dt \\
&=& \displaystyle \lim_{\delta \searrow 0} \left[ \int_\epsilon^{T-\epsilon} \eta(t) \int_{\T^3} \left( a_\delta(x,t) - \int_{\T^3} a_\delta(y,t) \, dy - \nabla g_\delta(x,t) \right) \cdot b_\delta(x,t) \, dx \, dt \right. \\
&+& \displaystyle \left. \int_\epsilon^{T-\epsilon} \eta(t) \int_{\T^3} \Psi_\delta(x,t) \cdot \nabla \times a_\delta(x,t) \, dx \, dt \right] \\
&=& \displaystyle 2 \lim_{\delta \searrow 0} \int_\epsilon^{T-\epsilon} \eta(t) \int_{\T^3} a_\delta(x,t) \cdot b_\delta(x,t) \, dx \, dt \\
&=& \displaystyle 2 \int_\epsilon^{T-\epsilon} \eta(t) \int_{\T^3} a(x,t) \cdot b(x,t) \, dx \, dt = 0,
\end{array}\]
since $a \cdot b = 0$.
\end{proof}

We present a variant of Theorem \ref{Magnetic helicity preservation theorem} which says that magnetic helicity is also conserved by weak limits of solutions. The proof is based on the observation that as a $\Lambda$-affine function, $Q(u,b,S,a) \defeq a \cdot b$ is weakly continuous for solutions of linearized 3D MHD.

\begin{lem} \label{Weak compactness lemma}
Suppose $b^j, u^j \in L^3(\T^3 \times ]0,T[;\R^3)$, $S^j \in L^1_{loc}(\T^3 \times ]0,T[; \mathcal{S}^{3 \times 3})$ and $a^j \in L^{3/2}(\T^3 \times ]0,T[;\R^3)$ satisfy the linearized MHD equations \eqref{MHD5}--\eqref{Linearized MHD2},\eqref{Linearized MHD 4}. Assume $b^j \rightharpoonup b$ and $u^j \rightharpoonup u$ in $L^3(\T^3 \times ]0,T[;\R^3)$ and $a^j \rightharpoonup a$ in $L^{3/2}(\T^3 \times ]0,T[;\R^3)$. Then $a^j \cdot b^j \to a \cdot b$ in $\mathcal{D}'(\T^3 \times ]0,T[)$.
\end{lem}

By a theorem of Tartar, every quadratic $\Lambda$-affine function is weakly continuous (see \cite[Corollary 13]{Tartar79}), but in \cite{Tartar79} the assumptions on the $L^p$ exponents are different (here $p = 3$ for $b$ and $p = 3/2$ for $a$) and the proof uses techniques of Fourier analysis that do not transfer immediately to our setting. We use instead the potentials provided by Lemma \ref{Time evolution of vector potential}.

\begin{proof}[Proof of Lemma \ref{Weak compactness lemma}]
Fix any subsequence of $a^j \cdot b^j$ (which we do not relabel); we show that there exists a further subsequence converging to $a \cdot b$ in $\mathcal{D}'(\T^3 \times ]0,T[)$. We use Lemma \ref{Time evolution of vector potential} to write $b^j = \nabla \times \Psi^j$ and $a^j - \int_{\T^3} a^j(y,\cdot) \, dy = -\partial_t \Psi^j + \nabla g^j$. We use the weak compactness of $L^p$ spaces in order to pass to a further subsequence and get $\Psi^j \rightharpoonup \Psi$ in $L^3(\T^3 \times ]0,T[;\R^3)$, $D \Psi^j \rightharpoonup M$ in $L^3(\T^3 \times ]0,T[;\R^{3 \times 3})$ and $\partial_t \Psi^j \rightharpoonup f$ in $L^{3/2}(\T^3 \times ]0,T[;\R^3)$, and clearly $M = D \Psi$ and $f = \partial_t \Psi$. By the Aubin-Lions compactness lemma (see e.g. \cite[Lemma 7.7]{Rou}), we conclude that $\Psi^j \to \Psi$ in $L^3(\T^3 \times ]\epsilon,T-\epsilon[;\R^3)$ whenever $0 < \epsilon < T$. We write
\begin{equation} \label{Emergence of Pfaffian}
a^j \cdot b^j = (-\partial_t \Psi^j + \nabla g^j) \cdot \nabla \times \Psi^j + \int_{\T^3} a^j(y,\cdot) \, dy \cdot \nabla \times \Psi^j
\end{equation}
and treat the two inner products separately. The first one is a sum of $L^1$-integrable compensated compactness quantities:
\[(-\partial_t \Psi^j + \nabla g^j) \cdot \nabla \times \Psi^j
= \frac{\partial(\Psi^j_1,\Psi^j_2)}{\partial (t,x_3)} + \frac{\partial(\Psi^j_2,\Psi^j_3)}{\partial (t,x_1)}
 + \frac{\partial(\Psi^j_3,\Psi^j_1)}{\partial (t,x_2)} + \nabla g^j \cdot \nabla \times \Psi^j,\]
and so $(-\partial_t \Psi^j + \nabla g^j) \cdot \nabla \times \Psi^j \to (-\partial_t \Psi + \nabla g) \cdot \nabla \times \Psi$ in $\mathcal{D}'(\T^3 \times ]0,T[)$. For the second one we fix $\eta \in C_c^\infty(\T^3 \times ]0,T[)$ and use Fubini's theorem, an integration by parts and the fact that $\Psi^j \to \Psi$ in $L^3_{loc}(\T^3 \times ]0,T[;\R^3)$ to get
\[\begin{array}{lcl}
& & \displaystyle \int_0^T \int_{\T^3} \eta(x,t) \int_{\T^3} a^j(y,t) \, dy \cdot \nabla_x \times \Psi^j(x,t) \, dx \, dt \\
&=& \displaystyle \int_0^T \int_{\T^3} \int_{\T^3} \nabla_x \eta(x,t) \times a^j(y,t) \cdot \Psi^j(x,t) \, dx \, dy \, dt \\
&\to& \displaystyle \int_0^T \int_{\T^3} \int_{\T^3} \nabla_x \eta(x,t) \times a(y,t) \cdot \Psi(x,t) \, dx \, dy \, dt \\
&=& \displaystyle \int_0^T \int_{\T^3} \eta(x,t) \int_{\T^3} a(y,t) \, dy \cdot \nabla_x \times \Psi(x,t) \, dx \, dt.
  \end{array}\]
Hence, $a^j \cdot b^j \to a \cdot b$ in $\mathcal{D}'(\T^3 \times ]0,T[)$.
\end{proof}

\begin{rem}
The expression $(-\partial_t \Psi^j + \nabla g^j) \cdot \nabla \times \Psi^j$ appearing in \eqref{Emergence of Pfaffian} is, up to a constant, a familiar compensated compactness quantity called the \emph{Pfaffian} of the antisymmetrized Jacobian matrix $2^{-1}(D_{x,t}-D^T_{x,t}) (\Psi,g) \colon \R^3 \times ]0,T[ \to \mathcal{A}^{4 \times 4}$.
\end{rem}

\begin{thm} \label{Conservation of magnetic helicity by weak limits of solutions}
Suppose $b^j, u^j \in L^3(\T^3 \times ]0,T[;\R^3)$ and $\Pi^j \in L^1_{loc}(\T^3 \times ]0,T[)$ satisfy the MHD equations \eqref{MHD}--\eqref{MHD3} and condition \eqref{MHD5}. Suppose $b^j \rightharpoonup b$ and $u^j \rightharpoonup u$ in $L^3(\T^3 \times ]0,T[;\R^3)$ and that $\Pi^j \rightharpoonup \Pi$ in $L^1_{loc}(\T^3 \times ]0,T[)$. Then $(b,u,\Pi)$ conserves magnetic helicity a.e. $t \in ]0,T[$.
\end{thm}

\begin{proof}
We denote $a^j \defeq b^j \times u^j \in L^{3/2}(\T^3 \times ]0,T[;\R^3)$ and $S^j \defeq u^j \otimes u^j - b^j \otimes b^j \in L^{3/2}(\T^3 \times ]0,T[;\mathcal{S}^{3 \times 3})$. Thus $(u^j,b^j,S^j,a^j)$ satisfy the linearized MHD equations \eqref{MHD5}--\eqref{Linearized MHD2} and \eqref{Linearized MHD 4}. By passing to a subsequence, $a^j \rightharpoonup a$ in $L^{3/2}(\T^3 \times ]0,T[;\R^3)$. If $\eta \in C_c^\infty(]0,T[)$, then by Lemma \ref{Weak compactness lemma} and the fact that $a^j \cdot b^j = 0$ for every $j \in \N$,
\[\int_0^T \partial_t \eta(t) \int_{\T^3} \Psi(x,t) \cdot b(x,t) \, dx
= 2 \int_0^T \eta(t) \int_{\T^3} a(x,t) \cdot b(x,t) \, dx \, dt
= 0.\]
\end{proof}

Theorem \ref{Conservation of magnetic helicity by weak limits of solutions} can also be proved by noting, as in the proof of Lemma \ref{Weak compactness lemma}, that $\Psi^j \to \Psi$ in $L^3$. However, Lemma \ref{Weak compactness lemma} is of independent interest and sheds light on the question whether the relaxation of 3D MHD coincides with the $\Lambda$-convex hull $K^\Lambda$ by showing that the relaxation respects the constraint $a \cdot b = 0$.

\section{Theorems \ref{Non-empty relative interior theorem}--\ref{Subsolution theorem} and estimates on $K^\Lambda$}
The aim of this section is to show the existence of non-trivial compactly supported strict subsolutions of 3D MHD. In \textsection \ref{Definition of subsolutions} we give the definition of strict subsolutions and formulate in Theorem \ref{Theorem on nonempty relative interior} a strengthening of Theorem \ref{Non-empty relative interior theorem} which says that the origin is in the relative interior of the lamination convex hull $K^{lc,\Lambda}$ (relative to the constraint given in Theorem \ref{Empty interior theorem on hull}), the proof is presented in \textsection \ref{A lemma on the relative interior of the hull}--\ref{Completion of the proof of hull lemma}. Theorem \ref{Subsolution theorem} is obtained as a rather straightforward corollary at the end of \textsection \ref{Definition of subsolutions}. We also briefly discuss the work of Bronzi, Lopes Filho and Nussenzweig Lopes on convex integration in 3D MHD in \textsection \ref{Discussion of the solutions of Bronzi & al.}.

\subsection{Definition of strict subsolutions and proof of Theorem \ref{Subsolution theorem}} \label{Definition of subsolutions}

Recall that in terms of the variables $u$, $b$, $S$ and $a$ we have $K_{r,s} = \{(u,b,S,a) \colon \abs{u+b}=r, \abs{u-b}=s, \abs{\Pi} \le rs, S = u \otimes u - b \otimes b + \Pi I, a = b \times u\}$. Theorem \ref{Empty interior theorem on hull} implies that if $(u,b,S,a) \in K^{lc,\Lambda}_{r,s}$, then $a \cdot b = 0$.

\begin{defin} \label{Definition of U}
Let $r,s > 0$. The relative interior of $K^{lc,\Lambda}_{r,s}$ in $\{(u,b,S,a) \in \R^3 \times \R^3 \times \mathcal{S}^{3 \times 3} \times \R^3 \colon a \cdot b = 0\}$ is denoted by $\mathcal{U}_{r,s}$.
\end{defin}

It will turn out in Theorem \ref{Theorem on nonempty relative interior} that $\mathcal{U}_{r,s} \neq \emptyset$. This motivates the following definition in analogy to, among others, the definition of subsolutions of Euler equations in \cite[p. 350]{DLS12}.

\begin{defin} \label{Subsolutions}
Let $r,s > 0$. The mappings $u,b \in L^2_{loc}(\R^3,\R^3)$, $S \in L^1_{loc}(\R^3,\mathcal{S}^{3 \times 3})$ and $a \in L^1_{loc}(\R^3,\R^3)$ form a \emph{strict subsolution} of \eqref{MHD}--\eqref{MHD3} if $(u,b,S,a)$ satisfies \eqref{Linearized MHD1}--\eqref{Linearized MHD2},\eqref{Linearized MHD 4} and $(u,b,S,a)(x,t) \in \mathcal{U}_{r,s}$ for almost every $(x,t) \in \R^3 \times \R$.
\end{defin}

The existence of strict subsolutions plays a pivotal part in the construction of convex integration solutions of equations of fluid dynamics in the Tartar framework. For the existence of strict subsolutions it is of course mandatory that the set $\mathcal{U}_{r,s}$ be non-empty. We record the following strengthening of Theorem \ref{Non-empty relative interior theorem}.

\begin{thm} \label{Theorem on nonempty relative interior}
Let $r,s > 0$. Then $0 \in \mathcal{U}_{r,s}$.
\end{thm}

Sections \ref{A lemma on the relative interior of the hull}--\ref{Completion of the proof of hull lemma} are devoted to the proof of Theorem \ref{Theorem on nonempty relative interior}. Assuming Theorem \ref{Theorem on nonempty relative interior} we now prove Theorem \ref{Subsolution theorem} via the existence of compactly supported strict subsolutions of Euler equations and the following simple result.

\begin{lem} \label{Lemma on 1D subsolutions}
Let $E \in C_c^\infty(\R^3 \times \R)$ and $\eta = (\eta',\eta_4) \in \R^3 \times \R$. Then
\begin{equation} \label{Formulas of b and a}
b \defeq \nabla E \times \eta', \qquad a \defeq -(\partial_t E) \eta' + \eta_4 \nabla E
\end{equation}
satisfy the conditions
\begin{align*}
& \dive b = 0, \\
& \partial_t b + \nabla \times a = 0, \\
& b \cdot a = 0.
\end{align*}
\end{lem}

\begin{proof}[Proof of Theorem \ref{Subsolution theorem}]
Let $r,s > 0$ and choose a solution $(u,S) \in C_c^\infty(\R^3 \times \R; \R^3 \times \mathcal{S}^{3 \times 3})$ of the linearized Euler equations $\partial_t u + \dive S = 0$ and $\dive u = 0$ with $u \not\equiv 0$ (see \cite[Lemma 4.4]{DLS09}). Then choose $E \in C_c^\infty(\R^3 \times \R)$ and $\eta = (\eta',\eta_4) \in \R^3 \times \R$ with $\nabla E \times \eta' \not\equiv 0$ and define $b,a \in C_c^\infty(\R^3 \times \R; \R^3)$ by \eqref{Formulas of b and a}. Now $(u,b,S,a) \in C_c^\infty(\R^3 \times \R; \R^3 \times \R^3 \times \mathcal{S}^{3 \times 3} \times \R^3)$, and after possibly multiplying by a small non-zero constant, Theorem \ref{Theorem on nonempty relative interior} gives $(u,b,S,a)(x,t) \in \mathcal{U}_{r,s}$ for every $(x,t) \in \R^3 \times \R$.
\end{proof}

\subsection{A lemma on the relative interior of the $\Lambda$-convex hull} \label{A lemma on the relative interior of the hull}
We begin the proof of Theorem \ref{Theorem on nonempty relative interior} by formulating in Lemma \ref{Hull lemma} a slightly stronger result which is also of independent interest. The lemma is stated in terms of Els\"asser variables, and we next discuss the relevant definitions in this formalism.

We first write the $\Lambda$-affine quantity $a \cdot b$ of Theorem \ref{Empty interior theorem on hull} in terms of Els\"{a}sser variables. When $\alpha,\beta \in \R^3$ and $\alpha \neq \beta$, we denote $f_1 \defeq (\alpha-\beta)/\abs{\alpha-\beta} = b/\abs{b}$ and suppose that $f_1$, $f_2$ and $f_3$ form an orthonormal basis of $\R^3$ with $f_1 \times f_2 = f_3$. If we write a general $3 \times 3$ matrix as $N \defeq \sum_{i,j=1}^3 c_{ij} f_i \otimes f_j$, then $a \cdot b = \abs{b} (c_{23} - c_{32})/2$. Indeed,
\begin{align*}
c_{23}-c_{32}
&= N f_3 \cdot f_2 - N f_2 \cdot f_3
 = (N - N^T) f_3 \cdot f_2 \\
&= 2 (A + u \otimes b - b \otimes u) f_3 \cdot f_2 \\
&= f_3 \times 2 (a + u \times b) \cdot f_2 \\
&= 2 (a + u \times b) \cdot f_2 \times f_3 \\
&= 2 a \cdot \frac{b}{\abs{b}} = 0.
\end{align*}
Theorem \ref{Theorem on nonempty relative interior} is a rather direct consequence of the following lemma once one sets $\tau = 0$, as we show after presenting the lemma.

\begin{lem} \label{Hull lemma}
If $r,s > 0$ and $0 \le \tau < 1$, then there exists a constant $c_{\tau, r, s} > 0$ with the following property: if

\renewcommand{\labelenumi}{(\roman{enumi})}
\begin{enumerate}

\item $\alpha, \beta \in \R^3$ and $\Pi \in \R$,

\item $\{f_1,f_2,f_3\}$ is an orthonormal basis of $\R^3$, where $f_1 = (\alpha-\beta)/\abs{\alpha-\beta}$ if $\alpha \neq \beta$,

\item $N = \sum_{i,j=1}^3 c_{ij} f_i \otimes f_j$ with $c_{23} = c_{32}$,

\item $\max\{\abs{\abs{\alpha}-\tau r}, \abs{\abs{\beta}-\tau s}, \abs{N}, \abs{\Pi} - \tau^2 rs\} < c_{\tau,r,s}$,
\end{enumerate}

\noindent then
\[(\alpha,\beta,\alpha \otimes \beta + \Pi I + N) \in K^{lc,\Lambda}_{r,s}.\]
\end{lem}

We indicate why Lemma \ref{Hull lemma} implies Theorem \ref{Theorem on nonempty relative interior}. Suppose that $(u,b,S,a) \cong (\alpha,\beta,\alpha \otimes \beta + \Pi I + N)$ is close to the origin and $a \cdot b = 0$. If $\alpha - \beta \neq 0$, then we write $N = \sum_{i,j=1}^3 c_{ij} f_i \otimes f_j$ and all the assumptions of Lemma \ref{Hull lemma} are satisfied for $\tau = 0$ so that $(u,b,S,a) \in K^{lc,\Lambda}_{r,s}$. If on the other hand $\alpha = \beta$, then there exists an orthonormal basis $\{f_1, f_2, f_3\}$ of $\R^3$ such that condition (iii) of Lemma \ref{Hull lemma} holds. Indeed, condition (iii) reads as $c_{23} - c_{32} = (N-N^T) f_3 \cdot f_2 = 0$, and writing $N - N^T \cong z \in \R^3$ we have $(N-N^T) f_3 = z \times f_3$, so that we get (iii) by choosing $f_3 = z/\abs{z}$ if $z \neq 0$ and $f_3 = (1,0,0)$ if $z = 0$. Thus the assumptions of Lemma \ref{Hull lemma} are satisfied and $(u,b,S,a) \in K^{lc,\Lambda}_{r,s}$. Lemma \ref{Hull lemma} is proved in the following two subsections.

\subsection{The case with $c_{23} = 0$} \label{The case of rank-one matrices}

We prove Lemma \ref{Hull lemma} by gradually weakening the assumptions on $N = \sum_{i,j=1}^3 c_{ij} f_i \otimes f_j$ in condition (iii). In this subsection we cover matrices $N$ with $c_{23} = 0$. First, in Lemma \ref{Restriction on next action} we give a geometric characterization of the pairs $V,W \in K$ such that $V-W \in \Lambda$. By iterating the idea we handle, in Lemma \ref{Easier directions}, the case where $N$ is a rank-one matrix, and in Lemma \ref{Five directions} a more elaborate use of $\Lambda$-convex combinations gives every $N$ with $c_{23} = 0$. The significantly harder case $c_{23} \neq 0$ is proved in \textsection \ref{Completion of the proof of hull lemma}.

\begin{lem} \label{Restriction on next action}
Suppose $\alpha, \beta, \gamma, \delta \in \R^3$ and $\Pi
_1,\Pi_2 \in \R$. Then the following conditions are equivalent.

\renewcommand{\labelenumi}{(\roman{enumi})}
\begin{enumerate}
\item $(\alpha, \beta, \alpha \otimes \beta + \Pi_1 I) - ( \gamma, \delta, \gamma \otimes \delta + \Pi_2 I) \in \Lambda$,

\item the four points $\alpha, \beta, \gamma$ and $\delta$ lie on the same hyperplane, that is, in a set of the form $\{x \in \R^3 \colon \langle x, \xi \rangle + c = 0\}$ with $\xi \neq 0$ and $c \in \R$. Furthermore, $\Pi_1 = \Pi_2$.
\end{enumerate}
\end{lem}

\begin{proof}[Proof]
Suppose $\alpha, \beta, \gamma, \delta \in \{x \in \R^3 \colon \langle x, \xi \rangle + c = 0\}$, where $\xi \in \R^3 \setminus \{0\}$, and $\Pi_1 = \Pi_2$. Then
\begin{equation} \label{Wave cone A}
\langle \alpha - \gamma, \xi \rangle = \langle \beta - \delta, \xi \rangle = 0,
\end{equation}
\begin{equation} \label{Wave cone B}
(\alpha \otimes \beta - \gamma \otimes \delta + \Pi_1 I - \Pi_2 I) \xi
= \langle \beta, \xi \rangle \alpha - \langle \delta, \xi \rangle \gamma
= \langle \delta, \xi \rangle (\alpha - \gamma) = -c(\alpha-\gamma),
\end{equation}
\begin{equation} \label{Wave cone C}
(\beta \otimes \alpha - \delta \otimes \gamma + \Pi_1 I - \Pi_2 I) \xi
= \langle \alpha, \xi \rangle \beta - \langle \gamma, \xi \rangle \delta
= \langle \gamma, \xi \rangle (\beta - \delta) = -c(\beta-\delta),
\end{equation}
and so $( \alpha, \beta, \alpha \otimes \beta + \Pi_1 I) - ( \gamma, \delta, \gamma \otimes \delta + \Pi_2 I)$ satisfies condition \eqref{Wave cone conditions}.

Conversely, if $( \alpha, \beta, \alpha \otimes \beta + \Pi_1 I) - ( \gamma, \delta, \gamma \otimes \delta + \Pi_2 I)$ satisfies \eqref{Wave cone conditions} with $(\xi,c) \in (\R^3 \setminus \{0\}) \times \R$, then \eqref{Wave cone A} follows immediately, and next $(\alpha \otimes \beta - \gamma \otimes \delta + (\Pi_1 - \Pi_2) I) \xi + c \alpha = 0$ implies $\Pi_1 = \Pi_2$, which in turn implies \eqref{Wave cone B}--\eqref{Wave cone C}, and so $\alpha, \beta, \gamma, \delta \in \{x \in \R^3 \colon \langle x, \xi \rangle + c = 0\}$.
\end{proof}

The following lemma covers the case $N = 0$ and gives us freedom in the selection of the points $\alpha,\beta \in \R^3$ in ensuing arguments.

\begin{lem} \label{Lemma for whole balls}
If $\abs{\alpha} \le r$, $\abs{\beta} \le s$ and $\abs{\Pi} \le rs$, then $(\alpha, \beta, \alpha \otimes \beta + \Pi I) \in K_{r,s}^{2,\Lambda}$.
\end{lem}

\begin{proof}
We prove the case $0 < \abs{\alpha} \le r$, $0 < \abs{\beta} \le s$, the remaining cases being similar. By Lemma \ref{Wave cone proposition 3}, the difference of the triples $(r \alpha/\abs{\alpha}, \pm s \beta/\abs{\beta}, r \alpha/\abs{\alpha} \otimes \pm s \beta/\abs{\beta} + \Pi I) \in K_{r,s}$ is in $\Lambda$, and so taking a suitable convex combination we get $(r \alpha/\abs{\alpha}, \beta, r \alpha/\abs{\alpha} \otimes \beta + \Pi I) \in K^{1,\Lambda}_{r,s}$. Similarly, $(-r \alpha/\abs{\alpha}, \beta, -r \alpha/\abs{\alpha} \otimes \beta + \Pi I) \in K^{1,\Lambda}_{r,s}$, and taking a $\Lambda$-convex combination again we obtain $(\alpha, \beta, \alpha \otimes \beta + \Pi I) \in K_{r,s}^{2,\Lambda}$.
\end{proof}

We next prove the case of rank-one matrices $N = c_{ij} f_i \otimes f_j$ with $\{i,j\} \neq \{2,3\}$ -- in fact, we prove a slightly more general statement. Recall that $f_1 = (\alpha-\beta)/\abs{\alpha-\beta}$ when $\alpha \neq \beta$.

\begin{lem} \label{Easier directions}
There exists a constant $c_{\tau,r,s}^\prime \in ]0,1[$ with the following property. If $\abs{\Pi} \le rs$, $\max\{\abs{\abs{\alpha} - \tau r}, \abs{\abs{\beta} - \tau s}, \abs{a}\} \le c_{\tau,r,s}^1$ and $e \in S^2$, then
\[(\alpha,\beta,\alpha \otimes \beta + a e \otimes e + \Pi I), (\alpha,\beta,\alpha \otimes \beta + a f_1 \otimes e + \Pi I), (\alpha,\beta,\alpha \otimes \beta + a e \otimes f_1 + \Pi I) \in K_{r,s}^{3,\Lambda}.\]
\end{lem}

\begin{proof}[Proof]
We first prove the case $N = a e \otimes e$. Choose $c_{\tau,r,s}^\prime \defeq (1-\tau)^2 rs/4(r+s+1)^2$. If $\max\{\abs{\abs{\alpha} - \tau r}, \abs{\abs{\beta} - \tau s}, \abs{a}\} \le c_{\tau,r,s}^\prime$ and $e \in S^2$, write $a = k c_{\tau,r,s}^\prime$, where $\abs{k} \le 1$. Then choose $b = k (1-\tau) r/2(r+s+1)$ and $c = d (1-\tau) s/2(r+s+1)$ so that $b c = a$. Now $\abs{\alpha \pm b e} \le r$ and $\abs{\beta \pm ce} \le s$, and so Lemma \ref{Lemma for whole balls} gives
\[V^\pm \defeq \left( \alpha \pm b e, \beta \pm c e, \left( \alpha \pm b e\right) \otimes (\beta \pm c e) + \Pi I \right) \in K_{r,s}^{2,\Lambda}.\]
The points $\alpha \pm b e$, $\beta \pm c e \in \R^3$ all belong to the same hyperplane (or straight) $\beta + \text{span} \{f_1,e\}$, and thus, by Lemma \ref{Restriction on next action}, $(V^++V^-)/2 = \left( \alpha,\beta, \alpha \otimes \beta + a e \otimes e + \Pi I \right) \in K_{r,s}^{3,\Lambda}$.

The case $N = a f_1 \otimes e$ is proved by setting
\[V^\pm \defeq \left( \alpha \pm b f_1, \beta \pm c e, \left( \alpha \pm b f_1\right) \otimes (\beta \pm c e) + \Pi I \right) \in K_{r,s}^{2,\Lambda}\]
and repeating the argument above. The case $N = a e \otimes f_1$ is similar.
\end{proof}

With the conclusion of Lemma \ref{Hull lemma} demonstrated for rank-one matrices $N$, we now prove the general case where $c_{23} = 0$. We present the result in an equivalent form that is easier to use in \textsection \ref{Completion of the proof of hull lemma}.

\begin{lem} \label{Five directions}
There exists $c_{\tau,r,s}^{\prime \prime} \in ]0,1[$ with the following property. Whenever $g_1,\ldots,g_5 \in \mathbb{S}^2$, $d_1,\ldots,d_5 \in \R$, $\max\{\abs{\abs{\alpha} - \tau r}, \abs{\abs{\beta} - \tau s}, \abs{d_1},\ldots, \abs{d_5}\} \le c_{\tau,r,s}^{\prime \prime}$ and $\abs{\Pi} \le rs$, we have 
\[(\alpha,\beta,\alpha \otimes \beta + d_1 g_1 \otimes g_1 + d_2 g_2 \otimes g_2 + d_3 g_3 \otimes g_3 + d_4 f_1 \otimes g_4 + d_5 g_5 \otimes f_1 + \Pi I) \in K_{r,s}^{7,\Lambda}.\]
\end{lem}

\begin{proof}
First assume $d_3 = d_4 = d_5 = 0$ and $\max\{\abs{\abs{\alpha} - \tau r}, \abs{\abs{\beta} - \tau s}, \abs{d_1},\abs{d_2}\} \le (c_{\tau,r,s}^\prime)^2/2$. Writing $b = \sqrt{\abs{d_2}}$ and $c = \text{sgn}(d_2) \sqrt{\abs{d_2}}$, Lemma \ref{Easier directions} gives
\[V^{\pm} \defeq (\alpha \pm b g_2, \beta \pm c g_2, (\alpha \pm b g_2) \otimes (\beta \pm c g_2) + d_1 g_1 \otimes g_1 + \Pi I) \in K^{3,\Lambda}_{r,s}\]
and Lemma \ref{Wave cone proposition 2} gives $V^+-V^- \in \Lambda$, so that $(\alpha, \beta, \alpha \otimes \beta + d_1 g_1 \otimes g_1 + d_2 g_2 \otimes g_2 + \Pi I) = (V^+ + V^-)/2 \in K^{4,\Lambda}_{r,s}$.

We next consider the general case. Set $c_{\tau,r,s}^{\prime \prime} \defeq (c_{\tau,r,s}^\prime)^2/16$ and suppose now that we have $\max\{\abs{\abs{\alpha} - \tau r}, \abs{\abs{\beta} - \tau s}, \abs{d_1},\ldots, \abs{d_5}\} \le c_{\tau,r,s}^{\prime \prime}$. By choosing $d_6 \neq 0$ with $\abs{d_6} \le c_{\tau,r,s}^{\prime \prime}$ and repeatedly using Lemma \ref{Easier directions} as above,
\[V \defeq (\alpha,\beta + d_6 f_1, \alpha \otimes (\beta + d_6 f_1) + d_1 g_1 \otimes g_1 + d_2 g_2 \otimes g_2 + d_3 g_3 \otimes g_3 + 2 d_4 f_1 \otimes g_4 + \Pi I),\]
\[W = (\alpha,\beta - d_6 f_1, \alpha \otimes (\beta - d_6 f_1) + d_1 g_1 \otimes g_1 + d_2 g_2 \otimes g_2 + d_3 g_3 \otimes g_3 + 2 d_5 g_5 \otimes f_1 + \Pi I)\]
belong to $K^{6,\Lambda}_{r,s}$. (Here we used the fact that $\alpha-(\beta \pm d_6 f_1)$ are parallel to $f_1$.) Lemma \ref{Wave cone proposition 3} gives
\[V-W = \left( 0, 2 d_6 f_1, \left( \alpha - \frac{d_5}{d_6} g_5 \right) \otimes 2 d_6 f_1 + 2 d_6 f_1 \otimes \frac{d_4}{d_6} g_4 \right) \in \Lambda,\]
and the claim follows by taking the average of $V$ and $W$.
\end{proof}

\subsection{Completion of the proof of Lemma \ref{Hull lemma}} \label{Completion of the proof of hull lemma}

One of the difficulties in proving the case $N = \sum_{i,j=1}^3 c_{ij} f_i \otimes f_j$ with $c_{23} = c_{32} \neq 0$ is the fact that the basis $\{f_1,f_2,f_3\}$ of $\R^3$ depends on $\alpha$ and $\beta$. We wish to use Lemma \ref{Five directions} to write
\[\begin{array}{lcl}
& & \displaystyle \left( \alpha,\beta,\alpha \otimes \beta + \sum_{i,j=1}^3 c_{ij} f_i \otimes f_j + \Pi I \right) \\
&=& \displaystyle \lambda V_1 + \mu V_2 \\
&\defeq& \displaystyle \lambda \left( \alpha + \mu f,\beta + \mu g,(\alpha + \mu f) \otimes (\beta + \mu g) + \sum_{i,j=1}^3 c_{ij}' f_i' \otimes f_j' + \Pi I \right) \\
&+& \displaystyle \mu \left( \alpha - \lambda f,\beta - \lambda g,(\alpha - \lambda f) \otimes (\beta - \lambda g) + \sum_{i,j=1}^3 c_{ij}^{\prime \prime} f_i^{\prime \prime} \otimes f_j^{\prime \prime} + \Pi I \right),
\end{array}\]
where $f_1'$ and $\alpha + \mu f - (\beta + \mu g)$ are parallel, $c_{23}' = c_{32}' = 0$, similar conditions hold for $V_2$ and furthermore $0 < \lambda < 1$, $\lambda + \mu = 1$ and $V_1 - V_2 \in \Lambda$. This leads, by necessity, to much more complicated computations than the ones done in \textsection \ref{The case of rank-one matrices}.

In particular, when $\alpha \neq \beta$, natural directions of the form $(f,g) = (c f_i, d f_j)$ with $c,d \in \R$ and $i,j \in \{1,2,3\}$ always lead to $c_{23} = c_{32} = 0$ unless $c,d \neq 0$ and $\{i,j\} = \{2,3\}$. With this specific choice of $(f,g)$ we are, however, eventually able to achieve $c_{23} = c_{32} \neq 0$. We finish the proof of Lemma \ref{Hull lemma} in two steps, starting with the following lemma.

\begin{lem} \label{Three non-vanishing terms}
There exists a constant $c_{\tau,r,s}^{\prime \prime \prime} \in ]0,1[$ such that if at most one of $c_{12}, c_{13}, c_{21}$ and $c_{31}$ nonzero and $\max\{\abs{\abs{\alpha} - \tau r}, \abs{\abs{\beta} - \tau s}, \abs{c_{ij}}, \abs{\Pi} - \tau^2 rs\} \le c_{\tau,r,s}^{\prime \prime \prime}$, then we have $(\alpha, \beta, \alpha \otimes \beta + \sum_{i,j=1}^3 c_{ij} f_i \otimes f_j + \Pi I) \in K^{9,\Lambda}_{r,s}$.
\end{lem}

\begin{proof}
We will find a constant $c_{\tau,r,s}^{\prime \prime \prime}$ for the case where $c_{13} = c_{21} = c_{31} = 0$; the modifications required for the other three cases are rather obvious. We will write $N = 2 N^\prime/3 + N^{\prime \prime}/3$, where $N^\prime \defeq \sum_{i,j=1}^3 c_{ij}' f_i \otimes f_j$ satisfies $c_{ij}' = c_{ij}$ for all $(i,j) \notin \{(1,2),(1,3)\}$ and the same condition holds for $N^{\prime \prime} \defeq \sum_{i,j=1}^3 c_{ij}^{\prime \prime}$, so that $(0,0,N'-N^{\prime \prime}) \in \Lambda$. In turn, we will write
\[(\alpha,\beta,\alpha \otimes \beta + N^\prime + \Pi I)
= \frac{1}{2} (V+W),\]
and similarly for $N^{\prime \prime}$, where $V,W \in K^{7,\Lambda}_{r,s}$ satisfy $V-W \in \Lambda$, fall under the scope of Lemma \ref{Five directions} and will be specified below.

We initially work under the full generality allowed by Lemma \ref{Five directions}. Recall that $\alpha-\beta = \abs{\alpha-\beta} f_1$. By Lemma \ref{Five directions}, whenever all the coefficients of $V$ and $W$ defined below are small enough, we have
\begin{align*}
V
&= (\alpha + c f_2, \beta + d f_3, (\alpha + c f_2) \otimes (\beta + d f_3) \\
&+ v_{f_1 \otimes f_1} f_1 \otimes f_1 + v_{f_2 \otimes f_2} f_2 \otimes f_2 + v_{f_3 \otimes f_3} f_3 \otimes f_3 \\
&+ (\abs{\alpha-\beta} f_1 + c f_2 - d f_3) \otimes (a_1 f_1 + a_2 f_2 + a_3 f_3) \\
&+ (b_1 f_1 + b_2 f_2 + b_3 f_3) \otimes (\abs{\alpha-\beta} f_1 + c f_2 - d f_3) + \Pi I) \in K^{7,\Lambda}_{r,s},
\end{align*}
\begin{align*}
W
&= (\alpha - c f_2, \beta - d f_3, (\alpha - c f_2) \otimes (\beta - d f_3) \\
&+ w_{f_1 \otimes f_1} f_1 \otimes f_1 + w_{f_2 \otimes f_2} f_2 \otimes f_2 + w_{f_3 \otimes f_3} f_3 \otimes f_3 \\
&+ (\abs{\alpha-\beta} f_1 - c f_2 + d f_3) \otimes (c_1 f_1 + c_2 f_2 + c_3 f_3) \\
&+ (d_1 f_1 + d_2 f_2 + d_3 f_3) \otimes (\abs{\alpha-\beta} f_1 - c f_2 + d f_3) + \Pi I) \in K^{7,\Lambda}_{r,s}.
\end{align*}
We will choose $c,d \neq 0$, and so, by Lemma \ref{Wave cone proposition 1} and the fact that
\[\alpha \otimes 2 d f_3 +  2 c f_2 \otimes \beta = \beta \otimes 2 d f_3 + 2 c f_2 \otimes \beta + \abs{\alpha-\beta} f_1 \otimes 2 d f_3,\]
the condition $V - W \eqdef (2c f_2,2d f_3, d_{f_1 \otimes f_1} f_1 \otimes f_1 + d_{f_1 \otimes f_2} f_1 \otimes f_2 + \cdots + d_{f_3 \otimes f_3} f_3 \otimes f_3) \in \Lambda$ reads as
\begin{align}
 d_{f_1 \otimes f_1} &= \abs{\alpha-\beta} (a_1 + b_1 - c_1 - d_1) + v_{f_1 \otimes f_1} - w_{f_1 \otimes f_1} = 0, \label{Wave cone condition 1} \\
d_{f_1 \otimes f_2} &= \abs{\alpha-\beta} (a_2-c_2) + c (b_1 + d_1) = 0, \label{Wave cone condition 2} \\
d_{f_3 \otimes f_1} &= \abs{\alpha-\beta} (b_3-d_3) - d(a_1 + c_1) = 0, \label{Wave cone condition 3} \\
d \cdot d_{f_2 \otimes f_1} &= cd(a_1 + c_1) + \abs{\alpha-\beta} d(b_2 - d_2) + 2 cd \langle \beta, f_1 \rangle \label{Wave cone condition 4} \\
&= c \cdot \abs{\alpha-\beta}(a_3-c_3+2d) -cd (b_1 + d_1) + 2 cd \langle \beta, f_1 \rangle = c \cdot d_{f_1 \otimes f_3} \label{Wave cone condition 5}
\end{align}
(recall that the relative lengths of the vectors $2c f_2$ and $2d f_3$ matter in condition \eqref{Wave cone condition 4}--\eqref{Wave cone condition 5}).

On the other hand, the sought equation $(\alpha,\beta,\alpha \otimes \beta + N' + \Pi I) = (V+W)/2$ can be written as
\[\begin{array}{lcl}
& & 2 \sum_{i,j=1}^3 c_{ij} f_i \otimes f_j + 2 (c_{12}' - c_{12}) f_1 \otimes f_2 + 2 (c_{13}'-c_{13}) f_1 \otimes f_3 \\
&=& [c (a_3-c_3) + d (d_2-b_2) + 2 cd] f_2 \otimes f_3 \\
&+& [d(c_2-a_2) + c(b_3-d_3)] f_3 \otimes f_2 \\
&+& [c (b_1-d_1) + \abs{\alpha-\beta} (a_2 + c_2)] f_1 \otimes f_2 \\
&+& [c (a_1-c_1) + \abs{\alpha-\beta} (b_2 + d_2)] f_2 \otimes f_1 \\
&+& [- d (b_1-d_1) + \abs{\alpha-\beta} (a_3 + c_3)] f_1 \otimes f_3 \\
&+& [- d (a_1-c_1) + \abs{\alpha-\beta} (b_3 + d_3)] f_3 \otimes f_1 \\
&+& [\abs{\alpha-\beta} (a_1 + b_1 + c_1 + d_1) + v_{f_1 \otimes f_1} + w_{f_1 \otimes f_1}] f_1 \otimes f_1 \\
&+& [c (a_2 + b_2 - c_2 - d_2) + v_{f_2 \otimes f_2} + w_{f_2 \otimes f_2}] f_2 \otimes f_2 \\
&+& [d (-a_3 - b_3 + c_3 + d_3) + v_{f_3 \otimes f_3} + w_{f_3 \otimes f_3}] f_3 \otimes f_3.
  \end{array}\]
Despite conditions \eqref{Wave cone condition 1}--\eqref{Wave cone condition 5} there is some freedom in the choice of the coefficients. We choose
\begin{align*}
V
&= (\alpha + c f_2, \beta + d f_3, (\alpha + c f_2) \otimes (\beta + d f_3) \\
&+ (c_{11} -q \abs{\alpha-\beta} - 2 p \abs{\alpha-\beta}^2) f_1 \otimes f_1 \\
&+ (c_{22} + c^2 p) f_2 \otimes f_2 + (c_{33} + (3 p - 1) d^2) f_3 \otimes f_3 \\
&+ (\abs{\alpha-\beta} f_1 + c f_2 - d f_3) \otimes (p \abs{\alpha-\beta} f_1 - c p f_2 + d (2 p - 1) f_3) \\
&+ ((p \abs{\alpha-\beta} + q) f_1 + d p f_3) \otimes (\abs{\alpha-\beta} f_1 + c f_2 - d f_3) + \Pi I) \\
&\in K^{7,\Lambda}_{r,s},
\end{align*}
\begin{align*}
W
&= (\alpha - c f_2, \beta - d f_3, (\alpha - c f_2) \otimes (\beta - d f_3) \\
&+ (c_{11} + q \abs{\alpha-\beta} - 2 p \abs{\alpha-\beta}^2) f_1 \otimes f_1 \\
&+ (c_{22} + c^2 p) f_2 \otimes f_2 + (c_{33} + (3 p - 1) d^2) f_3 \otimes f_3 \\
&+ (\abs{\alpha-\beta} f_1 - c f_2 + d f_3) \otimes (p \abs{\alpha-\beta} f_1 + c p f_2 - d (2 p - 1) f_3) \\
&+ ((p \abs{\alpha-\beta}-q) f_1 - d p f_3) \otimes (\abs{\alpha-\beta} f_1 - c f_2 + d f_3) + \Pi I)  \\
&\in K^{7,\Lambda}_{r,s},
\end{align*}
where the coefficients $c,d,p,q \in \R$ will be specified at the end of the proof. The wave cone conditions \eqref{Wave cone condition 1}--\eqref{Wave cone condition 5} are satisfied and thus
\begin{align*}
   \frac{V+W}{2}
&= (\alpha, \beta, \alpha \otimes \beta + c_{11} f_1 \otimes f_1 + c q f_1 \otimes f_2 - d q f_1 \otimes f_3 + c_{22} f_2 \otimes f_2 \\
&+ 2 p cd (f_2 \otimes f_3 + f_3 \otimes f_2) + c_{33} f_3 \otimes f_3 + \Pi I) \\
&\eqdef (\alpha,\beta,\alpha \otimes \beta + N' + \Pi I) \in K^{8,\Lambda}_{r,s}.
\end{align*}
We repeat the argument with $c$ and $d$ replaced by $-c/2$ and $-2d$ to form
\[N^{\prime \prime} = c_{11} f_1 \otimes f_1 - \frac{c}{2} q f_1 \otimes f_2 + 2 d q f_1 \otimes f_3 + c_{22} f_2 \otimes f_2
+ 2 p c d (f_2 \otimes f_3 + f_3 \otimes f_2) + c_{33} f_3 \otimes f_3,\]
and then $(0,0,N' - N^{\prime \prime}) = (0,0,(3cq/2) f_1 \otimes f_2 - 3 d q f_1 \otimes f_3) \in \Lambda$ implies
\[\begin{array}{lcl}
& & \displaystyle \frac{2}{3} (\alpha, \beta, \alpha \otimes \beta + N' + \Pi I)
+ \frac{1}{3} (\alpha, \beta, \alpha \otimes \beta + N^{\prime \prime} + \Pi I) \\
&=& \displaystyle (\alpha, \beta, \alpha \otimes \beta + c_{11} f_1 \otimes f_1 + \frac{c}{2} q f_1 \otimes f_2 + c_{22} f_2 \otimes f_2 \\
&+& 2 p cd (f_2 \otimes f_3 + f_3 \otimes f_2) + c_{33} f_3 \otimes f_3 + \Pi I) \in K^{9,\Lambda}_{r,s}.
  \end{array}\]

We are now ready to specify all the parameters. Set $c_{\tau,r,s}^{\prime \prime \prime} \defeq (c_{\tau,r,s}^{\prime \prime})^4/1000(r+s+1)^8$ and denote $c_{12} = k c_{\tau,r,s}^{\prime \prime \prime}$ with $\abs{k} \le 1$ and $c_{23} = l c_{\tau,r,s}^{\prime \prime \prime}$ with $\abs{l} \le 1$. Choose $c = (c_{\tau,r,s}^{\prime \prime \prime})^{1/2}$, $d = (c_{\tau,r,s}^{\prime \prime \prime})^{1/4}$, $p = l(c_{\tau,r,s}^{\prime \prime \prime})^{1/4}/2$ and $q = 2 k (c_{\tau,r,s}^{\prime \prime \prime})^{1/2}$, so that $c q/2 = c_{12}$, $2 p cd = c_{23}$. This proves the case $c_{13} = c_{21} = c_{31} = 0$. When forming $N^{\prime \prime} \in \R^{3 \times 3}$, if we replace $c$ and $d$ by $-2c$ and $-d/2$ (instead of $-c/2$ and $-2d$ as above) and modifying $p$ and $q$ accordingly, we get the case $c_{13} \neq 0$. The cases $c_{21} \neq 0$ and $c_{31} \neq 0$ are similar by symmetry.
\end{proof}

We finish the proof of Lemma \ref{Hull lemma} by reducing the general case $N = \sum_{i,j=1}^3 c_{ij} f_i \otimes f_j$ with $c_{23} = c_{32}$ to the situation of Lemma \ref{Three non-vanishing terms}.

\begin{proof}[Completion of the proof of Lemma \ref{Hull lemma}]
Suppose $(\alpha,\beta,\alpha \otimes \beta + \sum_{i,j=1}^3 c_{ij} f_i \otimes f_j + \Pi I)$ satisfies the assumptions of Lemma \ref{Hull lemma}, where $c_{\tau,r,s} \defeq c_{\tau,r,s}^{\prime \prime \prime}/8$. We write
\[\sum_{i,j=1}^3 c_{ij} f_i \otimes f_j \eqdef N' + c_{12} f_1 \otimes f_2 + c_{13} f_1 \otimes f_3 + c_{21} f_2 \otimes f_1 + c_{31} f_3 \otimes f_1\]
with the intention of using Lemma \ref{Three non-vanishing terms} and convex combinations in directions of the form $(0,0,M) \in \Lambda$ to prove that $(\alpha,\beta,\alpha \otimes \beta + \sum_{i,j=1}^3 c_{ij} f_i \otimes f_j + \Pi I) \in K^{12,\Lambda}_{r,s}$. We first form a $\Lambda$-convex combination via
\[\begin{array}{lcl}
& & c_{12} f_1 \otimes f_2 + c_{13} f_1 \otimes f_3 + c_{21} f_2 \otimes f_1 + c_{31} f_3 \otimes f_1 \\
&=& \displaystyle \frac{1}{2} (2 c_{12} f_1 \otimes f_2 + c_{21} f_2 \otimes f_1 + c_{31} f_3 \otimes f_1) \\
&+& \displaystyle \frac{1}{2} (2 c_{13} f_1 \otimes f_2 + c_{21} f_2 \otimes f_1 + c_{31} f_3 \otimes f_1) \\
&=& \displaystyle \frac{1}{2} \left( \frac{1}{2} (4 c_{12} f_1 \otimes f_2 + c_{31} f_3 \otimes f_1) + \frac{1}{2} (2 c_{21} f_2 \otimes f_1 + c_{31} f_3 \otimes f_1) \right) \\
&+& \displaystyle \frac{1}{2} \left( \frac{1}{2} (4 c_{13} f_1 \otimes f_3 + c_{21} f_2 \otimes f_1) + \frac{1}{2} (c_{21} f_2 \otimes f_1 + 2 c_{31} f_3 \otimes f_1) \right).
\end{array}\]
Note that $(\alpha,\beta, \alpha \otimes \beta + N' + 2 c_{21} f_2 \otimes f_1 + c_{31} f_3 \otimes f_1 + \Pi I)$ is, furthermore, a $\Lambda$-convex combination of $(\alpha,\beta, \alpha \otimes \beta + N' + 4 c_{21} f_2 \otimes f_1 + \Pi I) \in K^{9,\Lambda}_{r,s}$ and $(\alpha,\beta, \alpha \otimes \beta + N' + 2 c_{31} f_3 \otimes f_1 + \Pi I) \in K^{9,\Lambda}_{r,s}$, and a similar argument works on $c_{21} f_2 \otimes f_1 + 2 c_{31} f_3 \otimes f_1$.

For the matrix $4 c_{12} f_1 \otimes f_2 + c_{31} f_3 \otimes f_1$ we get
\[(\alpha,\beta,\alpha \otimes \beta + N' + 4 c_{12} f_1 \otimes f_2 + c_{31} f_3 \otimes f_1 + \Pi I)
= \frac{V+W}{2} \in K^{10,\Lambda}_{r,s}\]
where, for a small enough  $d \in \mathbb{R} \setminus \{0\}$, we have
\begin{align*}
V &\defeq (\alpha,\beta + d f_1,\alpha \otimes (\beta + d f_1) + N' + 8 c_{12} f_1 \otimes f_2 + \Pi I) \in K^{9,\Lambda}_{r,s}, \\
W &\defeq (\alpha,\beta - d f_1,\alpha \otimes (\beta - d f_1) + N' + 2 c_{31} f_3 \otimes f_1 + \Pi I) \in K^{9,\Lambda}_{r,s}
\end{align*}
by Lemma \ref{Three non-vanishing terms} and $V-W \in \Lambda$ by Lemma \ref{Wave cone proposition 3}. Analogous reasoning shows that $(\alpha,\beta,\alpha \otimes \beta + N' + 4 c_{13} f_1 \otimes f_3 + c_{21} f_2 \otimes f_1 + \Pi I) \in K^{10,\Lambda}_{r,s}$, and as a result, $(\alpha,\beta,\alpha \otimes \beta + \sum_{i,j=1}^3 c_{ij} f_i \otimes f_j + \Pi I) \in K^{12,\Lambda}_{r,s}$.
\end{proof}

We have now proved Lemma \ref{Hull lemma} and, thereby, Theorems \ref{Non-empty relative interior theorem}, \ref{Subsolution theorem} and \ref{Theorem on nonempty relative interior}. We finish the article by making some brief remarks on the convex integration solutions of \cite{BLFNL}.

\subsection{Discussion of the solutions of Bronzi \& al.} \label{Discussion of the solutions of Bronzi & al.}

As mentioned in the Introduction, Bronzi, Lopes Filho and Nussenzweig Lopes showed in \cite{BLFNL} the existence of bounded weak solutions of 3D MHD that are compactly supported in time. In this subsection we briefly discuss their construction.

Bronzi \& al$.$ studied two-dimensional incompressible Euler equations with a passive tracer,
\begin{align}
& \partial_t v + \dive (v \otimes v) + \nabla p = 0, \label{Passive tracer 1} \\
& \partial_t b + \dive (b v) = 0, \label{Passive tracer 2} \\
& \dive v = 0, \label{Passive tracer 3}
\end{align}
where $v \colon \R^2 \times \R \to \R^2$ is the velocity field, $b \colon \R^2 \times \R \to \R$ is the tracer and $p \colon \R^2 \times \R \to \R$ is the pressure. Their main result, proved by adapting arguments of \cite{DLS09} on Euler equations, reads as follows:

\begin{thm}[\cite{BLFNL}] \label{Theorem of the Brazilians}
Given a bounded domain $\Omega \subset \R^2 \times \R$, there exists a weak solution $(v,b) \in L^\infty(\R^2 \times \R; \R^2 \times \R)$ of \eqref{Passive tracer 1}--\eqref{Passive tracer 3} with the following properties:
\renewcommand{\labelenumi}{(\roman{enumi})}
\begin{enumerate}
\item $\abs{v(x,t)} = 1$ and $\abs{b(x,t)} = 1$ for almost every $(x,t) \in \Omega$,

\item $v(x,t) = 0$, $b(x,t) = 0$ and $p(x,t) = 0$ for almost every $(x,t) \in (\R^2 \times \R) \setminus \Omega$.
\end{enumerate}

\end{thm}

Bronzi \& al$.$ obtained bounded weak solutions of 3D MHD from Theorem \ref{Theorem of the Brazilians} as follows: when $u,b \colon \R^3 \times \R \to \R^3$ are of the symmetry reduced forms
\begin{equation} \label{Symmetry restrictions}
u(x_1,x_2,x_3,t) \defeq (u_1(x_1,x_2,t), u_2(x_1,x_2,t),0), \; b(x_1,x_2,x_3,t) = (0,0,b_3(x_1,x_2,t)),
\end{equation}
the 3D MHD equations \eqref{MHD}--\eqref{MHD3} are reduced to \eqref{Passive tracer 1}--\eqref{Passive tracer 3}, and so Theorem \ref{Theorem of the Brazilians} yields solutions of \eqref{MHD}--\eqref{MHD3}. However, as mentioned in the Introduction, solutions $u$ and $b$ obtained in this way are not compactly supported in space.

The main obstacle to finding compactly supported solutions of 3D MHD is the non-linear constraint $a \cdot b = 0$. By using Lemma \ref{Lemma on 1D subsolutions} one may show that a single localized 1D wave can take values in $\mathcal{U}_{r,s}$, but the localization makes it difficult to superimpose waves while satisfying the condition $a \cdot b = 0$. In \cite{BLFNL} this issue did not arise simply because Bronzi \& al$.$ imposed the symmetry restrictions \eqref{Symmetry restrictions} which lead (in this article's notation) to $a$ and $b$ taking values in orthogonal subspaces of $\R^3$. It seems that going beyond the symmetry reduced solutions of \cite{BLFNL} and attaining compact support in space will require very careful analysis.

Note also that the symmetry restrictions \eqref{Symmetry restrictions} ensure that the cross helicity $\int_{\R^3} u(x,t) \cdot b(x,t)$ vanishes identically for the solutions of \cite{BLFNL}. It is therefore still an open question whether cross helicity is conserved by weak solutions of 3D MHD.

\bigskip
\footnotesize
\noindent\textit{Acknowledgments.}
The authors thank Diego C\'{o}rdoba for suggesting the problem of studying convex integration in MHD. We also thank \'{A}ngel Castro, Sara Daneri and L\'{a}szl\'{o} Sz\'{e}kelyhidi Jr. for discussions on the article.

\end{document}